\documentclass[a4paper,10pt]{article}
\usepackage{lipsum}
\usepackage{amsmath}
\usepackage{emptypage}
\usepackage{hyperref}
\hypersetup{colorlinks=true,
	linkcolor=blue,
	filecolor=magenta,     
	urlcolor=cyan,
	}
\usepackage[a4paper,top=28mm,bottom=32mm,left=28mm,right=28mm]{geometry}
\usepackage[toc,page]{appendix}
\usepackage{enumerate}
\usepackage[utf8]{inputenc}
\usepackage[english]{babel}
\usepackage{amsfonts}
\usepackage{amssymb}
\usepackage{makecell}
\usepackage{amsthm}
\usepackage{array}
\usepackage{amsmath}
\usepackage{pgfplots}
\usepackage{keytheorems}

\pgfplotsset{compat=1.18}
\usepackage{cases}
\usepackage{tikz}
\newcommand{\g}{{\mathfrak g}}

\newcommand{\aaa}{{\mathfrak a}}
\newcommand{\nil}{{\mathfrak {nil}}}

\newcommand{\R}{{\mathbb R}}
\newcommand{\N}{{\mathbb N}}
\newcommand{\Q}{{\mathbb Q}}
\newcommand{\C}{{\mathbb C}}
\newcommand{\G}{{\mathcal {G}}}

\newcommand{\h}{{\mathfrak h}}

\newcommand{\n}{{\mathfrak n}}
\DeclareMathOperator{\im}{Im}
\DeclareMathOperator{\Id}{Id}
\DeclareMathOperator{\Aut}{Aut}
\DeclareMathOperator{\Spec}{Spec}
\DeclareMathOperator{\Span}{Span}
\DeclareMathOperator{\End}{End}
\DeclareMathOperator{\sgn}{sgn}
\DeclareMathOperator{\Der}{Der}

\DeclareMathOperator{\Dg}{Dg}

\DeclareMathOperator{\ad}{ad}
\DeclareMathOperator{\Car}{Car}
\DeclareMathOperator{\Tr}{Tr}
\DeclareMathOperator{\GL}{GL}

\newkeytheorem{theorem}[parent=section]
\newkeytheorem{lemma}[sibling=theorem]
\newkeytheorem{proposition}[sibling=theorem]
\newkeytheorem{corollary}[sibling=theorem]
\newkeytheorem{claim}[sibling=theorem]
\newkeytheorem{definition}[sibling=theorem]
\newkeytheorem{example}[sibling=theorem]
\newkeytheorem{question}[sibling=theorem]
\newkeytheorem{notation}[sibling=theorem]
\newkeytheorem{terminology}[sibling=theorem]
\newkeytheorem{remark}[sibling=theorem]
\usepackage{cleveref}
\title{Nice bases for Lie algebras}
\date{\vspace{-5ex}}
\author{Jonas Deré\footnote{KU Leuven Kulak, Kortrijk, Belgium.  \href{mailto: jonas.dere@kuleuven.be}{jonas.dere@kuleuven.be}.} \and Jeroen Gantois \footnote{KU Leuven Kulak, Kortrijk, Belgium. \href{mailto: jeroen.gantois@kuleuven.be}{jeroen.gantois@kuleuven.be}. Research supported by Methusalem grant METH/21/03 long-term structural funding of the Flemish government.}}
\begin{document}
	\maketitle
\begin{abstract}
The concept of a nice basis for a Lie algebra was introduced to study the Ricci curvature on nilpotent Lie groups equipped with a left-invariant metric. Despite the many applications in differential geometry, for example in the construction of Einstein manifolds, very little is known about the existence and number of nice bases on a given Lie algebra. This paper studies this question for three classes of Lie algebras, namely direct sums, almost abelian ones and nilpotent Lie algebras associated to a graph. As an application we compute the number of nice bases for Lie algebras up to dimension \(3\), and show that for a general Lie algebra the existence depends on the field over which it is defined. Moreover, for every natural number \(n\) we give an indecomposable Lie algebra such that there exist exactly \(n\) nice bases up to equivalence.
\end{abstract}
\section{Introduction}
	The Ricci curvature tensor of a connected Lie group equipped with a left-invariant metric can be expressed in terms of its associated metric Lie algebra. Although the formula is rather involved, the existence of a nice basis for the Lie algebra considerably simplifies calculations. In \cite[Lemma 3.9]{Lauret-Will-first-mention} the concept of a nice basis was introduced, without the name which originated only later. Recall that if \(\{X_1, \ldots, X_n\}\) is a basis for a Lie algebra \(\g\), then its structure constants \(c_{ij}^k\) are given by
	\begin{align*} [X_i,X_j] = \sum_{k=1}^n c_{ij}^k X_k.
	\end{align*}
		\begin{definition}
		A basis \(\{X_1,\dots,X_n\}\) for a Lie algebra \(\g\) is \textit{nice} if the following two conditions on the structure constants \(c_{ij}^k\) are satisfied.
		\begin{enumerate}[(1)]
			\item For all \(i, j \in \{1, \ldots, n\}\) there exists at most one \(k \in \{1, \ldots, n\}\) such that \(c_{ij}^k\neq 0.\)
			\item If for any \(i, j, k, s, t \in \{1, \ldots, n\}\) it holds that if \(c_{ij}^k\neq 0\) and \(c_{st}^k\neq 0\), then either \(\{i,j\}=\{s,t\}\) or \(\{i,j\}\cap\{s,t\}=\emptyset.\)
		\end{enumerate}
	\end{definition} 
\noindent Note that the definition of a nice basis can be given for Lie algebras over an arbitrary field. Unless mentioned otherwise, we consider Lie algebras over the field of real numbers \(\R\).
	
	The paper \cite{Lauret-Will-first-mention} shows that a nice basis for a nilpotent Lie algebra is always stably Ricci-diagonal, meaning that any inner product for which the basis is orthogonal has a diagonal Ricci tensor. In \cite{Lauret-Will-6dim} the converse was proven, i.e.~any stably Ricci-diagonal basis for a nilpotent Lie algebra must be nice. The definition of a nice basis depends only on the Lie algebra itself, so determining which nilpotent Lie algebras have a stably Ricci diagonal basis becomes an algebraic problem. More recently, it has been shown in \cite{Krishnan-21} that stably Ricci diagonal bases for semisimple Lie algebras of compact type correspond to nice bases, motivating the study of nice bases beyond the class of nilpotent Lie algebras.
	
	The existence of a nice basis has many applications throughout the literature. For example, in \cite{Milnor} Milnor implicitly used the fact that any \(3\)-dimensional unimodular Lie algebra has a nice basis to classify the possible signatures for the Ricci curvature of the associated Lie groups. They are also used to study Ricci flow, see \cite{Lauret-Will-first-mention,PayneTracy}, and Ricci negative Lie groups as in \cite{Deré-Lauret,Lauret-Will-RN}. Moreover, \cite{Nikolayevsky} demonstrates how nice bases can be used to study Einstein solvmanifolds. Despite the many applications, it is non-trivial to characterize the Lie algebras admitting such a basis. 
	
	 Many nilpotent Lie algebras have a natural basis that is nice, for example, in \cite[Proposition 2.1]{Lauret-Will-6dim} it is shown that there is only one nilpotent Lie algebra up to dimension \(6\) that does not have one. In \cite[Examples 3, 4]{Nikolayevsky} and \cite[Theorem 3.11]{Conti-Rossi-Ad-invariant}, more examples of nilpotent Lie algebras not admitting a nice basis are provided. However, examples of nilpotent Lie algebras not admitting a nice basis remain scarce. Note that if \(\mathcal{B}\) is a nice basis for a Lie algebra \(\g\), then we can construct another nice basis from \(\mathcal{B}\) by multiplying the elements of \(\mathcal{B}\) with nonzero constants. This motivates the definition of an equivalence relation on the set of nice bases of a given Lie algebra \(\g\), as introduced in \cite{Conti-Rossi}.
		\begin{definition}
		We say that two nice bases \(\mathcal{B}_1\) and \(\mathcal{B}_2\) for a Lie algebra \(\g\) are equivalent if there exists a Lie algebra automorphism sending elements of \(\mathcal{B}_1\) to multiples of elements of \(\mathcal{B}_2\).
	\end{definition}
	\noindent The equivalence relation defined above will be denoted by `\(\sim\)' and the number of equivalence classes by \(\nu_\g\). This all motivates the question of describing nice bases on a Lie algebra up to equivalence.
	
		\begin{question}
		Given a Lie algebra \(\g\), how many nice bases \(\nu_\g\) does \(\g\) have up to equivalence?
	\end{question}
	
\noindent Despite the extensive literature on examples, surprisingly little is known about the structural behavior of nice bases themselves. For instance, it is not understood how nice bases behave under natural algebraic constructions or how strongly their existence depends on the underlying field. The main results of this paper deal with these questions.
	
	In \cite{Conti-Rossi} it is proven that this number can be at most countable for finite dimensional real nilpotent Lie algebras. Later it was shown in \cite{Gorbatsevich} that the number is finite for these Lie algebras. Moreover, in the same paper it has been shown that the number of nice bases for a Lie algebra can be arbitrary large by arguing that 
	\[\bigoplus_{i=1}^{2k}\h_3\]
	has at least \(k+1\) nice bases up to equivalence with \(\h_3\) the \(3\)-dimensional Heisenberg Lie algebra. The exact number of nice bases is unknown in general for these direct sums. There are also known examples of indecomposable nilpotent Lie algebras with multiple nice bases, see \cite[Proposition 3.1]{Conti-Rossi}. This raises the question whether we can find for every natural number \(n\) an indecomposable Lie algebra which has exactly \(n\) nice bases up to equivalence, for which we will present a positive answer.

	 In \cite{Conti-Rossi} an algorithm is introduced to classify the nilpotent Lie algebras admitting a nice basis for a fixed dimension. The output is presented as a table containing the \(n\)-dimensional nilpotent Lie algebras with a nice basis defined by their structure equations. If a Lie algebra admits \(k\) nice bases up to equivalence, the algorithm prints this Lie algebra \(k\) times, once for each non-equivalent nice basis. Therefore, the number of nilpotent Lie algebras that the algorithm gives us is not equal to the number of \(n\)-dimensional nilpotent Lie algebras which admit a nice basis. Furthermore, if you want to determine for a specific \(n\)-dimensional nilpotent Lie algebra how many nice bases it admits, you should check whether each entry in the table produced by the algorithm is isomorphic to your given Lie algebra. As these tables get increasingly larger if the dimension increases, this task becomes more difficult for high dimensions. Furthermore this algorithm is useless when studying infinite families of nilpotent Lie algebras.
	 
	 The first results deal with the structure of nice bases under taking direct sums, as discussed in \Cref{sect. direct-sum}. Any Lie algebra can be written as a direct sum of indecomposable ideals. \Cref{thm. |g|=|g+Rn|} shows that if a Lie algebra \(\g\) has a nice basis and an abelian ideal \(\aaa\) in its decomposition, then \(\g/\aaa\) has a nice basis as well. For more general factors, we use the \textit{pre-Einstein derivation} to prove the following. \getkeytheorem{Pre-einstein-directsum}
	 
The aim of \Cref{sect. Almost-abelian} is to address two questions which we mentioned earlier. More specifically we will show that the existence of a nice basis depends on the field over which you are working and that there exists for every \(n\in\N\) and every field \(K\subseteq \C\) an indecomposable Lie algebra over \(K\) with exactly \(n\) nice bases. For this we consider the class of almost abelian Lie algebras, which was already briefly mentioned in \cite{Gorbatsevich}. \Cref{thm. total-nice-basis-semidirect} provides a characterization of the almost abelian Lie algebras admitting a nice basis, whereas \Cref{thm. number-NB-semidirect} gives a description of the number of nice bases for these Lie algebras. In \Cref{sect. 3-dim} we will use the results obtained in the previous section to provide an overview of the number of nice bases for all \(3\)-dimensional Lie algebras.
	 
	 Finally, we will consider nilpotent Lie algebras associated to graphs in \Cref{sect. graphs}. Graph Lie algebras constitute a natural class of nilpotent Lie algebras whose structure constants are already governed by combinatorial data. Originally, these were introduced only in nilpotency class \(2\) in \cite{Mainkar05} to study Anosov diffeomorphisms, but they were also used in the context of nilmanifolds in \cite{DeCosteDeMeyerMainkar18,Pouseele-Tirao09}. Later, they were studied for arbitrary nilpotency classes in \cite{Dere-Witdouck24,Dere-Witdouck23,Mainkar06}. In \cite{Lauret-Will-first-mention} the two step nilpotent Lie algebras associated to graphs which are Einstein nilradicals are discussed. In general, the study of these nilradicals is challenging. However, the existence of a nice basis for all two step nilpotent Lie algebras associated to graphs significantly simplifies the analysis of these objects. As mentioned before, this fact is discussed in more generality in \cite{Nikolayevsky}. Our final result characterizes the nilpotent Lie algebras associated to graphs having a nice basis.
	 \getkeytheorem{NB-graphs}
	 This result will be used in upcoming work to generate many classes of Einstein nilradicals using \cite[Theorem 3]{Nikolayevsky}, contributing to a deeper understanding of these objects.
	\section{Nice bases for direct sum Lie algebras}\label{sect. direct-sum}
	Any finite-dimensional Lie algebra \(\g\) can be written as a direct sum of indecomposable ideals, say
\[\g=\g_1\oplus\g_2\oplus\dots\oplus\g_s\]
for some \(s\in\N.\) Moreover, it has been shown in \cite{Fisher} that this decomposition is unique up to reordering of the factors. If each of the \(\g_i\) admits a nice basis, it follows immediately that \(\g\) admits a nice basis as well. A natural question is whether the converse holds as well. More specifically, if there is a \(\g_i\) which does not admit a nice basis, does this imply that \(\g\) does not admit a nice basis? For this it suffices to restrict to the study of direct sums of two (not necessarily indecomposable) Lie algebras.
\begin{question}\label{question}
	Does the existence of a nice basis for \(\g\oplus \h\) imply the existence of a nice basis for \(\g\) and \(\h\)? Moreover, can we describe \(\nu_{\g\oplus\h}\) in terms of \(\nu_{\g}\) and \(\nu_{\h}\)?\end{question}	
Before studying this question we will introduce some notation.
\begin{notation}
 Suppose that \(\g\) and \(\h\) are arbitrary Lie algebras with bases \(\mathcal{B}_1=\{X_1,\dots,X_n\}\) and \(\mathcal{B}_2=\{Y_1,\dots,Y_m\}\) respectively, then
		\[\{(X_1,0),\dots,(X_n,0),(0,Y_1),\dots,(0,Y_m)\}\]
		forms a basis for the direct sum Lie algebra \(\g\oplus\h.\) For simplicity of notation, we will denote this new basis by \(\mathcal{B}_1\oplus\mathcal{B}_2\). 
\end{notation}
	The answer to Question \ref{question} is positive if \(\g\) or \(\h\) are abelian. An important tool that we will use in the proof of this result is the following. The terms of the upper and lower central series of a Lie algebra \(\g\) will be denoted by \(Z_i(\g)\) and \(\gamma_j(\g)\) respectively, where \(i\in\N_{\geq 1}\) and  \(j\in\N_{\geq 2}\).
	
	\begin{lemma}\label[lemma]{lem. UCS/LCS-adapted}
		If a Lie algebra \(\g\) admits a nice basis \(\mathcal{B}\), then \(Z_i(\g)\) and \(\gamma_j(\g)\) are spanned by \(\mathcal{B}\cap Z_i(\g)\) and \(\mathcal{B}\cap\gamma_j(\g)\) respectively.
	\end{lemma}
This lemma for the lower central series was first given in \cite{Lauret-Will-6dim} and for the upper central series  in \cite{Conti-Rossi-Ad-invariant}. We say that nice bases are \textit{adapted} to the upper and lower central series of the Lie algebra.
	\begin{theorem}\label{thm. |g|=|g+Rn|}
		Let \(\g\) be a Lie algebra and let \(\aaa\) be an abelian Lie algebra. If \(\g\oplus\aaa\) admits a nice basis then \(\g\) admits a nice basis as well. Moreover, it holds that
		\[\nu_{\g\oplus\aaa}=\nu_\g.\]
	\end{theorem}
	\begin{proof}
		Suppose that \(\dim(\g)=n\) and \(\dim(\aaa)=m.\) Now let \[\mathcal{B}=\{(Y_1,Z_1),\dots,(Y_{n+m},Z_{n+m})\}\] be a nice basis for \(\g\oplus\aaa\). It then follows from \Cref{lem. UCS/LCS-adapted} that \[[\g\oplus\aaa,\g\oplus\aaa]=[\g,\g]\oplus\{0\}\] is spanned by elements of \(\mathcal{B}\). If \(\dim([\g,\g])=t\), then reordering the elements if necessary gives us \(Z_1=\dots=Z_t=0\) and
		\[[\g,\g]=\Span\{Y_1,\dots,Y_t\}.\]
		Now extend \(\{Y_1,\dots,Y_t\}\) to a basis of \(\g\) by adding the vectors \(Y_{t+1},\dots,Y_{n}\), which is possible after possibly reordering the elements in the basis \(\mathcal{B}\). By using that \(\mathcal{B}\)  is a nice basis for \(\g\oplus\aaa\) it follows that \(\{Y_1,\dots,Y_n\}\) must be a nice basis for \(\g.\), giving the first part of the theorem.

		For the second part we use the same notations as before. First we show that every nice basis for \(\g\oplus\aaa\) is equivalent to a basis of the form
		\[\mathcal{B}_1\oplus\mathcal{B}_2,\]
		where \(\mathcal{B}_1\) is a nice basis for \(\g\) and \(\mathcal{B}_2\) is any basis for \(\aaa.\) We already know from the above that if 
		\[\mathcal{B}=\{(Y_1,Z_1),\dots,(Y_{n+m},Z_{n+m})\}\]
		is a nice basis for \(\g\oplus\aaa\), then we can assume that \(\mathcal{B}_1:=\{Y_1,\dots,Y_n\}\) is a nice basis for \(\g.\) We then find from \Cref{lem. UCS/LCS-adapted} that \(\{Y_1,\dots,Y_n\}\) contains exactly \(\dim(Z(\g))\) central elements. This implies that
		\[\{(Y_1,Z_1),\dots,(Y_{n},Z_{n})\}\]
		contains exactly \(\dim(Z(\g))\) central elements of \(\g\oplus\aaa\). As \(\mathcal{B}\) is a nice basis for \(\g\oplus\aaa\), \Cref{lem. UCS/LCS-adapted} implies that \(Z(\g\oplus\aaa)\) is spanned by \(\dim(Z(\g))+m\) elements of \(\mathcal{B}\) and thus \(Y_{n+j}\in Z(\g)\) for all \(j\in\{1,\dots,m\}.\) 
		
		Now let \(\mathcal{B}_2:=\{U_1,\dots,U_m\}\) be an arbitrary basis for \(\aaa\). A straightforward computation shows that the linear map \(\varphi:\g\oplus\aaa\rightarrow \g\oplus\aaa\) defined by
		\[\varphi((Y_i,Z_i))=(Y_i,0),\quad \varphi((Y_{n+j},Z_{n+j}))=(0,U_j)\]
		for \(i\in\{1,\dots,n\}\) and \(j\in\{1,\dots,m\}\) is a Lie algebra automorphism sending elements of \(\mathcal{B}\) to multiples of elements of the basis \(\mathcal{B}_1\oplus\mathcal{B}_2,\) which we wanted to prove.

		To end the proof, let \(\{V_1,\dots,V_n\}\) and \(\{W_1,\dots,W_n\}\) be nice bases for \(\g\), then it suffices to show that \(\{V_1,\dots,V_n\}\sim\{W_1,\dots,W_n\}\) if and only if 
		\[\mathcal{B}_1=\{(V_1,0),\dots,(V_n,0),(0,U_1),\dots,(0,U_m)\}\sim\mathcal{B}_2= \{(W_1,0),\dots,(W_n,0),(0,U_1),\dots,(0,U_m)\}.\]
		The \textit{`only if'-}part is immediately clear. Conversely, suppose that \(\varphi\in\Aut(\g\oplus\aaa)\) maps elements of \(\mathcal{B}_1\) to multiples of elements of \(\mathcal{B}_2.\) As \(\{V_1,\dots, V_n\}\) and \(\{W_1,\dots,W_n\}\) are nice bases, we find by \Cref{lem. UCS/LCS-adapted} that there exists a \(k\in\{0,\dots,n\}\) such that (after possibly reordering)
		\[V_1,\dots,V_k,W_1,\dots,W_k\]
		belong to \(Z(\g)\setminus[\g,\g].\)
		Now define the linear map \(\psi:\g\rightarrow\g\) on basis elements as follows
		\[\psi(V_i)=\begin{cases}
	W_i	&\text{ if }i\leq k \\
				\pi(\varphi(V_i,0)) &\text{ if }i> k
		\end{cases}.\]
		Note that for all \(i>k\) it holds that \(V_i\) belongs to \([\g,\g]\) or to \(\g\setminus Z(\g).\) This implies that 
		\[\varphi(V_i,0)=\alpha(W_j,0)\]
		for some \(j>k\) and \(\alpha\in\R_0\). We can thus conclude from this that \(\psi\) is a vector space automorphism of \(\g\), mapping elements of \(\{V_1,\dots,V_n\}\) to multiples of elements of \(\{W_1,\dots,W_n\}\). Note that if \(i,j>k\) then \([V_i,V_j]=\beta V_t\) for some \(\beta\in\R\) and \(t>k,\) this implies that
		\[\psi([V_i,V_j])=\beta\psi(V_t)=\beta\pi(\varphi(V_t,0))=[\pi(\varphi(V_i,0)),\pi(\varphi(V_j,0))]=[\psi(V_i),\psi(V_j)].\]
		If \(i\) or \(j\) lies in \(\{1,\dots,k\}\), it follows immediately that \(\psi([V_i,V_j])=[\psi(V_i),\psi(V_j)]\), so we can conclude that \(\psi\) is a Lie algebra automorphism of \(\g.\) This proves that \(\{V_1,\dots,V_n\}\) and \(\{W_1,\dots,W_n\}\) are equivalent, which we wanted to show.  
	\end{proof}
	From the proof above we can find the following corollary.
	\begin{corollary}
		Let \(\g\) be a Lie algebra and let \(\aaa\) be an abelian Lie algebra. Any nice basis \(\mathcal{B}\) for the direct sum Lie algebra \(\g\oplus\aaa\) is equivalent to a basis of the form
		\[\mathcal{B}_1\oplus\mathcal{B}_2,\]
		where \(\mathcal{B}_1\) is a nice basis for \(\g\) and \(\mathcal{B}_2\) is a (nice) basis for \(\aaa\). 
	\end{corollary}
	\begin{remark} Note that \cite[Proposition 3.3]{Conti-Rossi} states that if a central extension of a Lie algebra admits a nice basis, then the Lie algebra itself has a nice basis as well. This would make the first part of the proof of \Cref{thm. |g|=|g+Rn|} redundant. However, below we will construct a counterexample for central extensions, hence the statement of \cite[Proposition 3.3]{Conti-Rossi} is false.
\end{remark}

\begin{example} Start from the six-dimensional nilpotent Lie algebra \(\n_6\) which does not admit a nice basis by \cite{Lauret-Will-6dim}. This Lie algebra has basis \(\{X_1,\dots,X_6\}\) and Lie bracket
		\[[X_1,X_2]=X_4,\quad[X_1,X_4]=X_5,\quad [X_1,X_5]=[X_2,X_3]=[X_2,X_4]=X_6.\]
		Now consider the 7-dimensional nilpotent Lie algebra \(\n\) with basis \(\{X_1,\dots,X_7\}\) and Lie bracket 
		\[[X_1,X_2]=X_4,\quad[X_1,X_4]=X_5,\quad [X_1,X_5]=[X_2,X_3]=X_6,\quad [X_2,X_4]=X_7.\]
	The basis \(\{X_1,\dots,X_7\}\) is nice. As \(Z(\n)=\Span\{X_6,X_7\}\), the Lie algebra \(\n\) is a central extension of \(\n_6\) because  
		\[\frac{\n}{\Span\{X_6-X_7\}}\cong \n_6.\]
		Note that we can also not expect a central extension of a Lie algebra with a nice basis to have a nice basis as well. Indeed, note that \(Z(\n_6)=\Span\{X_6\}\) and \(\frac{\n_6}{Z(\n_6)}\) both admit a nice basis, but \(\n_6\) does not admit a nice basis.
\end{example}
For general Lie algebra \(\g\) and \(\h\), we provide an answer to Question \ref{question} under additional assumptions. For this we need the \textit{pre-Einstein derivation} (also called the \textit{Nikolayevsky derivation}) of a Lie algebra, which was first introduced in \cite{Nikolayevsky}. It is shown in \cite{Nikolayevsky} that every nilpotent Lie algebra admits a real semisimple derivation \(N\in\Der(\g)\) such that for all \(D\in\Der(\g)\) it holds that
	\[\Tr(ND)=\Tr(D).\]
	Moreover it is shown in \cite{Nikolayevsky} that \(N\) is unique up to conjugation with a Lie algebra automorphism of \(\g\). From this it follows immediately that the set of eigenvalues of the pre-Einstein derivation of a nilpotent Lie algebra is uniquely determined. 
	Although not explicitly stated, one can deduce from the proof of \cite[Theorem 3]{Nikolayevsky} the following result. 
	\begin{proposition}\label[proposition]{prop. NB->diag. Nik.}
		Let \(\g\) be a Lie algebra which admits a nice basis, then there exists a pre-Einstein derivation of \(\g\) which is diagonal with respect to this nice basis. Equivalently, if \(\g\) admits a nice basis, we can find for any pre-Einstein derivation \(N\) an equivalent nice basis for \(\g\) such that \(N\) is diagonal with respect to this basis.
	\end{proposition}
	 We will provide a simple argument for this result, in the proof we will use the following lemma.
	 \begin{lemma}
	 	If \(\g\) is a Lie algebra with a nice basis, then it holds for all \(D\in\Der(\g)\) that the diagonal of \(D\) in this basis is again a derivation of \(\g.\)
	 \end{lemma} 
	 \begin{proof}
	 	This is proven in \cite[Corollary 3.15]{Deré-Lauret}. 
	 \end{proof}
	\begin{proof}[Proof of \Cref{prop. NB->diag. Nik.}]
		Note that both statements are equivalent since the pre-Einstein derivation is unique up to conjugation by a Lie algebra automorphism. 

		Suppose that \(\mathcal{B}=\{X_1,\dots,X_n\}\) and define \(\Dg(\g)\) to be the set of linear operators on \(\g\) which are diagonal with respect to \(\mathcal{B}\). The bilinear map 
		\[\left(\Der(\g)\cap\Dg(\g)\right)\times \left( \Der(\g)\cap\Dg(\g)\right)\rightarrow \R: (D_1,D_2)\mapsto \Tr(D_1D_2)\]
		defines an inner product on \(\Der(\g)\cap\Dg(\g)\). Hence there exists a unique \(N\in \Der(\g)\cap\Dg(\g)\) such that \(\Tr(ND')=\Tr(D' )\) for all \(D' \in \Der(\g)\cap\Dg(\g)\).
		
		We now claim that \(N\) is a pre-Einstein derivation for \(\g\). For this, pick any \(D\in\Der(\g)\) and consider its diagonal \(\Dg(D)\) with respect to the basis \(\{X_1, \ldots, X_n\}\). As \(\mathcal{B}\) is a nice basis, \(\Dg(D)\) is again a derivation of \(\g\) which we denote by \(D'\). We thus have that \(D'\in\Der(\g)\cap\Dg(\g)\). The following equalities now hold
		\[\Tr(ND)=\Tr(ND')=\Tr(D')=\Tr(D).\]
		This concludes the proof.
	\end{proof}
	As an example, we will use the method of the above proof to calculate the pre-Einstein derivation for standard filiform Lie algebras.
		\begin{example}\label[example]{ex. Nik-Der-L_n}
		Pick \(n\in\N_{\geq 3}\) and let \(L_n\) denote the \(n\)-dimensional standard filiform Lie algebra. This is the Lie algebra with basis \(\{e_1,\dots,e_n\}\) and Lie bracket
		\[[e_1,e_i]=e_{i+1},\quad\forall i\in\{2,\dots,n-1\}.\]
		As this basis is nice it follows from the proof of \Cref{prop. NB->diag. Nik.} that it suffices to find a derivation \(N\), diagonal with respect to this basis, such that for all other diagonal derivations \(D'\) it holds that
		\begin{equation}\label{eq. pre-Einst-eqn}
		\Tr(ND')=\Tr(D').
		\end{equation}
		Note that any diagonal derivation with respect to this basis is of the form
		\[\Dg(x,y,x+y,2x+y,\dots,(n-2)x+y),\]
		for \(x,y\in\R.\) If we let
		\[N^{(n)}=\Dg(d_1^{(n)},d_2^{(n)},d_1^{(n)}+d_2^{(n)},2d_1^{(n)}+d_2^{(n)},\dots,(n-2)d_1^{(n)}+d_2^{(n)})\]
		be a pre-Einstein derivation of \(L_n,\) we find from condition \eqref{eq. pre-Einst-eqn} that
		\[xd_1^{(n)}+yd_2^{(n)}+\sum_{k=1}^{n-2}(kx+y)(kd_1^{(n)}+d_2^{(n)})=x+y+\sum_{k=1}^{n-2}(kx+y),\]
		for all \(x,y\in\R.\) By equating the coefficients of \(x\) and \(y\) and using the sum formulas for the first \(n-2\) natural numbers and squares, we obtain a linear system in the variables \(d_1^{(n)}\) and \(d_2^{(n)}.\) This system has a unique solution, namely
		\[d_1^{(n)}=\frac{12}{n^3-3n^2+2n+12},\qquad d_2^{(n)}=\frac{n^3-3n^2-4n+24}{n^3-3n^2+2n+12}.\]
	\end{example}
	
\medskip	

Using \Cref{prop. NB->diag. Nik.}, we can give another proof for the fact that \(\n_6\) does not admit a nice basis. Note that in contrast to the proof of \cite[Proposition 2.1]{Lauret-Will-6dim}, our argument does not require to distinguish specific Lie algebras up to isomorphism.

	\begin{proposition}
		The Lie algebra \(\n_6\) with basis \(\{X_1,\dots,X_6\}\) and Lie bracket 
		\[[X_1,X_2]=X_4,\quad [X_1,X_4]=X_5,\quad [X_1,X_5]=[X_2,X_3]=[X_2,X_4]=X_6\]
		does not admit a nice basis.
	\end{proposition}
	\begin{proof}
		The pre-Einstein derivation of \(\n_6\) with respect to the given basis is of the form
		\[N=\frac{9}{32}\Dg(1,2,3,3,4,5).\]
		This expression follows from computing the derivation algebra of \(\n_6\). Note that it is not immediate that this derivation is diagonal with respect to the given basis (as it is not a nice basis).
	
		Suppose by contradiction that \(\n_6\) does admit a nice basis, then we find by the previous proposition a nice basis \(\mathcal{B}=\{X'_1,\dots,X'_6\}\) with respect to which \(N\) is diagonal. We thus find without loss of generality that \[X_1'\in\Span\{X_1\},\quad X_2'\in\Span\{X_2\},\quad X_5'\in\Span\{X_5\},\quad X_6'\in\Span\{X_6\}.\] Since \(\mathcal{B}\) is a nice basis it holds that \([X_1',X_2']\) is a multiple of a nice basis element. On the other hand we know that \([X_1',X_2']\) must be a nonzero element of \(\Span\{X_4\}\), we can thus suppose that \([X_1',X_2']\) is a nonzero multiple of \(X_4'\in\Span\{X_4\}\). From this it follows that \(X_3'=\alpha X_3+\beta X_4\), with \(\alpha, \beta \in \R\) and \(\alpha\neq 0.\) 
		
		From the discussion above we now find that there exists a \(\gamma,\delta\in\R_0\) such that \(X_2'=\gamma X_2\) and  \(X_4'=\delta X_4\), so 
		\[[X_2',X_3']=\alpha\gamma[X_2,X_3]+\beta\gamma[X_2,X_4]=\gamma(\alpha+\beta)X_6\in\Span\{X_6'\}\]
		\[[X_2',X_4']=\gamma\delta X_6 \in\Span\{X_6'\}\setminus\{0\}.\]
		The above is only possible if \(\alpha=-\beta,\) so \(X_3'=\alpha(X_3-X_4)\). We now find from this that
		\[[X_1',X_3']\in\Span\{X_5'\}\setminus\{0\},\]
		however we also have that \([X_1',X_4']\in\Span\{X_5'\}\setminus\{0\}\) which contradicts the definition of a nice basis. We conclude that \(\n_6\) does not admit a nice basis.
	\end{proof}
This leads to the main result of this section.
\begin{theorem}[store=Pre-einstein-directsum]\label[theorem]{thm. NB-g+h-Nik}
		Let \(\g\) and \(\h\) be Lie algebras with pre-Einstein derivations \(N_1\) and \(N_2\) such that \(\Spec(N_1)\cap\Spec(N_2)=\emptyset\). If \(\g\oplus\h\) admits a nice basis, then so do \(\g\) and \(\h\), and moreover
		\[\nu_{\g\oplus\h}=\nu_\g\cdot \nu_\h.\]
\end{theorem}
	\begin{proof}
		Let \(n:=\dim(\g)\) and \(m:=\dim(\h).\)
		Note that \[N = N_1\oplus N_2\] is a pre-Einstein derivation of \(\g\oplus\h\) by \cite[Theorem 7]{Nikolayevsky}. As \(\g\oplus\h\) admits a nice basis \(\mathcal{B}=\{X_1,\dots,X_{n+m}\}\), we can assume that \(N\) is diagonal with respect to \(\mathcal{B}\). By reordering the elements, we get that \(\{X_1,\dots,X_n\}\) is a set of eigenvectors of \(N\) with eigenvalues in \(\Spec(N_1)\) and \(\{X_{n+1},\dots,X_{n+m}\}\) is a set of eigenvectors of \(N\) with eigenvalues in \(\Spec(N_2)\). Since \(\Spec(N_1)\cap\Spec(N_2)=\emptyset\), it follows that 
	\[\{X_1,\dots,X_n\},\qquad \{X_{n+1},\dots,X_{n+m}\}\]
	are nice bases for \(\g\) and \(\h\) respectively.

	By the previous discussion we know that any nice basis for \(\g\oplus\h\) is of the form \(\mathcal{B}_1\oplus\mathcal{B}_2\), where \(\mathcal{B}_1\) and \(\mathcal{B}_2\) are nice bases for \(\g\) and \(\h\) respectively, such that \(N_1\) and \(N_2\) are diagonal with respect to these bases. To prove the second part it therefore suffices to show that \[\mathcal{B}_1\oplus\mathcal{B}_2\sim\mathcal{B}_1'\oplus\mathcal{B}_2'\iff\mathcal{B}_1\sim\mathcal{B}_1'\text{ and }\mathcal{B}_2\sim\mathcal{B}_2'.\]
	 Note that the \textit{`if'-}part is trivial. For the converse, suppose that \(\mathcal{B}_1\oplus\mathcal{B}_2\sim\mathcal{B}_1'\oplus\mathcal{B}_2'\). We can then find a map \(\psi\in\Aut(\g\oplus\h)\) such that \(\psi\) maps elements from \(\mathcal{B}_1\oplus\mathcal{B}_2\) to multiples of elements of \(\mathcal{B}_1'\oplus\mathcal{B}_2'\). If \(X\in\mathcal{B}_1\) it holds that \(\psi(X)\) cannot be a multiple of an element of \(\mathcal{B}_2' \). Indeed, \(\psi(X)\) is an eigenvector of a pre-Einstein derivation with corresponding eigenvalue in \(\Spec(N_1)\) and as  \(\Spec(N_1)\cap\Spec(N_2)=\emptyset\) we conclude that \(\psi(X)\) must be a multiple of an element in \(\mathcal{B}_1' \). We thus have that
	 \[\psi=\psi_1\oplus\psi_2,\]
	 where \(\psi_1\in\Aut(\g)\) and \(\psi_2\in\Aut(\h)\). This proves the equivalence. 
	\end{proof}
	From the proof of the theorem above, we find the following.
\begin{corollary}
	Let \(\g\) and \(\h\) be Lie algebras with respective pre-Einstein derivations \(N_1\) and \(N_2\). If \(\Spec(N_1)\cap\Spec(N_2)=\emptyset\), then it follows that any nice basis for \(\g\oplus\h\) is equivalent to
	\[\mathcal{B}_1\oplus\mathcal{B}_2,\]
	where \(\mathcal{B}_1\) is a nice basis for \(\g\) and \(\mathcal{B}_2\) is a nice basis for \(\h\).
\end{corollary}
By applying \Cref{thm. NB-g+h-Nik} multiple times, we find the following corollary.
	\begin{corollary}\label[corollary]{cor. directsum-Nik-Der}
		Let \(\g_1, \g_2,\dots,\g_n\) be Lie algebras with respective pre-Einstein derivations \(N_1, N_2, \dots, N_n\) such that the spectra of these derivations are pairwise disjoint. If the Lie algebra \(\g:=\bigoplus_{i=1}^n\g_i\) admits a nice basis, then so do \(\g_1, \g_2, \dots, \g_n\). Moreover \[\nu_{\g}=\prod_{i=1}^n\nu_{\g_i}.\]
	\end{corollary}
	
\begin{remark}
	Note that we cannot expect the formula \(\nu_{\g\oplus\h}=\nu_\g\cdot\nu_\h\) to hold in general since it is shown in \cite[Proposition 3.1]{Conti-Rossi} that \(\nu_{\h_3\oplus\h_3}=2\) but \(\nu_{\h_3}=1.\) Moreover, as mentioned in the introduction, it is known that the direct sum of \(2k\) copies of \(\h_3\) has at least \(k+1\) nice bases up to equivalence. At this point, we want to emphasize that \(k+1\) is only a lower bound. For example, it can be shown that the direct sum of \(4\) copies of \(\h_3\) has at least \(4\) nice bases. Hence, for a given Lie algebra \(\g\),  an exact formula for \(\nu_{\g^{\oplus n}}\) currently is out of reach.
\end{remark}

As mentioned in \cite{Lauret-Will-6dim}, the existence of a simple derivation (i.e.~a real semisimple derivation with pairwise distinct eigenvalues) is a sufficient condition for a Lie algebra to admit a nice basis. By combining this with a similar argument as in the proof of \Cref{thm. NB-g+h-Nik}, we find the following corollary.
\begin{corollary}\label[corollary]{cor. suf-cond-1NB}
	If the pre-Einstein derivation of a Lie algebra \(\g\) is simple, then \(\g\) has a unique nice basis up to equivalence. 
\end{corollary}
	We will now apply \Cref{thm. NB-g+h-Nik} in the following proposition.
	\begin{proposition}\label[proposition]{prop. n_6+L_n_no_nice_basis}
			Pick \(n\in\N_{0}\) arbitrary and choose \(i_1,\dots,i_n\in\N_{\geq 3}\). It holds that the Lie algebra \[\g:=\n_6\oplus\bigoplus_{k=1}^nL_{i_k}\] does not admit a nice basis.
		\end{proposition}
		\begin{proof}
		The pre-Einstein derivation \(N_1\) of \(\n_6\) with respect to the usual basis is of the form
		\[\frac{9}{32}\Dg(1,2,3,3,4,5).\] Denote by \(N_2\) the pre-Einstein derivation of \[\bigoplus_{k=1}^nL_{i_k}.\] We will show that \(\Spec(N_1)\cap\Spec(N_2)=\emptyset\), for this it suffices to show that \(\Spec(N_1)\cap \Spec(N^{(n)})=\emptyset\) for all \(n\geq 3\). We provided plots for some eigenvalues of the pre-Einstein derivation of \(L_n\) as a function of \(n\) in the figure below.
		
		Note that for all \(n\geq 5\), \(d_1^{(n)}<9/32\) which is the smallest element of \(\Spec(N_1)\). It can be checked explicitly that for \(n=3,4\) it also holds that \(d_1^{(n)}\notin\Spec(N_1).\) On the other hand note that for all \(n\geq 7\) it holds that 
		\[3\cdot\frac{9}{32}=\frac{27}{32}<d_2^{(n)}<\frac{9}{8}=4\cdot\frac{9}{32}.\]
		Note that \(27/32\) and \(9/8\) are consecutive eigenvalues of \(N_1\) and moreover it can be checked explicitly that for all \(n=3,4,5,6\) we have that \(d_2^{(n)}\notin\Spec(N_1)\).
		
		Now note that for all \(n\in\N_{\geq 3}\)
		\[d_1^{(n)},d_2^{(n)}> 0.\]
		From this it follows that \(d_1^{(n)}+d_2^{(n)}<kd_1^{(n)}+d_2^{(n)}<(n-2)d_1^{(n)}+d_2^{(n)}\) for all \(k\in\{2,\dots,n-3\}\). Now note that for all \(n\in\N_0\) it holds that \(d_1^{(n)}+d_2^{(n)}>27/32\) and for all \(n\geq 8\) it holds that \((n-2)d_1^{(n)}+d_2^{(n)}<9/8\). We thus find that for all \(n\geq 8\) and \(k\in \{2,\dots,n-3\}\) it holds that 
		\[3\cdot\frac{9}{32}=\frac{27}{32}<d_1^{(n)}+d_2^{(n)}<kd_1^{(n)}+d_2^{(n)}<(n-2)d_1^{(n)}+d_2^{(n)}<\frac{9}{8}=4\cdot\frac{9}{32}.\]
		We can conclude from this that for all \(n\in \N_{\geq 8}\) we have that \(\Spec(N_1)\cap\Spec\left(N^{(n)}\right)=\emptyset.\)
		It can be checked explicitly that the same holds for \(n=3,4,5,6,7.\)
		
		We can now conclude by contradiction. If \(\g\) admits a nice basis, we would find by \Cref{thm. NB-g+h-Nik} that \(\n_6\) admits a nice basis as well, which is impossible.
		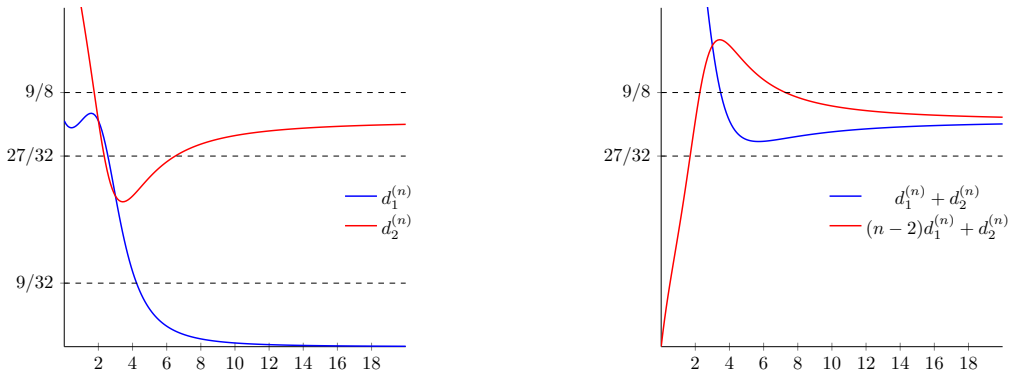
\begin{figure}[h]\label{fig. EV-L_n}
			
			\centering
			
			\begin{minipage}{0.45\textwidth} 
				\centering
			\begin{tikzpicture}[scale=0.7]
				\begin{axis}[
					axis lines=middle,
					axis line style={-},
					xmin=0, xmax=20,
					ytick={9/32,27/32,9/8},
					yticklabels={\(9/32\),\(27/32\),\(9/8\)},
					xtick={2,4,6,8,10,12,14,16,18},
					ymin=0, ymax=1.5,
					width=8cm,
					height=8cm,
					legend style={ at={(1.05,0.5)}, draw=none} 
					]
					\addplot[thick, blue,domain=0:20,samples=200] {(12)/(x^3-3*x^2+2*x+12)};
					\addlegendentry{\(d_1^{(n)}\)}
					\addplot[thick,red,domain=0:20,samples=200] {(x^3-3*x^2-4*x+24)/(x^3-3*x^2+2*x+12)};
					\addlegendentry{\(d_2^{(n)}\)};
					\addplot[dashed,domain=0:20,samples=200]{27/32} ;
					\addplot[dashed,domain=0:20,samples=200]{9/8};
					\addplot[dashed,domain=0:20,samples=200] {9/32};	
				\end{axis}
			\end{tikzpicture}
			\end{minipage}
			\hspace{0.05\textwidth}
			\begin{minipage}{0.45\textwidth}
				\centering
					\begin{tikzpicture}[scale=0.7]
					\begin{axis}[
						axis lines=middle,
						axis line style={-},
						xmin=0, xmax=20,
						ytick={27/32,9/8},
						yticklabels={\(27/32\),\(9/8\)},
						xtick={2,4,6,8,10,12,14,16,18},
						ymin=0, ymax=1.5,
						width=8cm,
						height=8cm,
						legend style={ at={(1.05,0.5)}, draw=none} 
						]
						\addplot[thick, blue,domain=0:20,samples=200] {(12)/(x^3-3*x^2+2*x+12)+(x^3-3*x^2-4*x+24)/(x^3-3*x^2+2*x+12)};
						\addlegendentry{\(d_1^{(n)}+d_2^{(n)}\)}
						\addplot[thick,red,domain=0:20,samples=200] {(12*(x-2))/(x^3-3*x^2+2*x+12)+(x^3-3*x^2-4*x+24)/(x^3-3*x^2+2*x+12)};
						\addlegendentry{\((n-2)d_1^{(n)}+d_2^{(n)}\)};
						\addplot[dashed,domain=0:20,samples=200] {9/8};
						\addplot[dashed,domain=0:20,samples=200] {27/32};	
					\end{axis}
				\end{tikzpicture}
			\end{minipage}
			\caption{Eigenvalues of the pre-Einstein derivation of  \(L_n\)}
		\end{figure}
	\end{proof}
	We can see from the proposition above that for all \(n>3\) it holds that
	\[0<d_1^{(n)}<d_2^{(n)}.\]
	This implies in particular that the eigenvalues of the pre-Einstein derivation of \(L_n\) are pairwise distinct for all \(n>3\). Since \(L_3 \cong\h_3\) has one nice basis up to equivalence, \Cref{cor. suf-cond-1NB} implies the following.
	\begin{theorem}
		For all \(n\in\N_{\geq 3}\) it holds that \(\nu_{L_n}=1.\)
	\end{theorem}
	\section{Almost abelian Lie algebras}\label{sect. Almost-abelian}
	In this section we will consider \textit{almost abelian Lie algebras}, these are Lie algebras with an abelian ideal of codimension one. Each such Lie algebra can be written as a semi-direct product \(\R f\ltimes_A\R^n\), where \(A\in\End(\R^{n})\). An arbitrary element in \(\R f\ltimes_A\R^n\) is of the form
	\[\alpha f+X\]
	where \(
	X\in\R^n\) and \(\alpha\in \R.\) The Lie bracket of two elements in \(\R f\ltimes_A\R^n\) is then given by
	\[[\alpha_1 f+X_1,\alpha_2f+X_2]=A(\alpha_1 X_2 -\alpha_2X_1).\]
	
A priori, a nice basis for \(\R f\ltimes_A\R^n\) can be of the form \(\mathcal{B}=\{\alpha_1f+X_1,\dots,\alpha_{n+1}f+X_{n+1}\}\). In the following proposition we will see that there is a restriction on the number of vectors in \(\mathcal{B}\) which do not belong to the ideal \(\R^n\).
	\begin{proposition}\label[proposition]{prop. form of nice basis RfxR^n}
		Every nice basis of \(\R f\ltimes_A\R^n\) is equivalent to a basis of the form \[\{f,X_1,\dots,X_n\},\]
		where the matrix of \(A\) with respect to \(\{X_1,\dots,X_n\}\) has at most one nonzero entry in each column and each row.
	\end{proposition}
	\begin{proof}
		Suppose that \(\R f\ltimes_A\R^n\) admits a nice basis \(\mathcal{B}=\{\alpha_1f+X_1,\dots,\alpha_nf+X_n,\alpha_{n+1}f+X_{n+1}\}\). We will work by case distinction to prove the result.
		
		First suppose that \(A^2\neq 0\). We will show that there is exactly one \(i\in\{1,\dots,n+1\}\) such that \(\alpha_i\neq 0.\) For this note that
		\[0 \neq \im(A) = [\g,\g] \subseteq \R^n,\]
		\[0 \neq \im(A^2) =  [\g,[\g,\g]] \subseteq \R^n.\]
		Since \([\g,\g]\) and \([\g,[\g,\g]]\) are spanned by elements of the nice basis by Lemma \ref{lem. UCS/LCS-adapted}, it must hold that there is an \(i\in\{1,\dots,n\}\) such that \(\alpha_i\) equals zero and \(X_i\in\im(A)\setminus\ker(A)\). Now suppose that \(\alpha_k\)
		and \(\alpha_l\) are both nonzero. It then follows that
		\[[\alpha_kf+X_k,X_i]=\alpha_kAX_i,\]
		\[[\alpha_lf+X_l,X_i]=\alpha_lAX_i,\]
		which is only possible if \(\alpha_k=0\) or \(\alpha_l=0\). On the other hand there must be at least one \(k\) such that \(\alpha_k\neq0\) as otherwise we do not have a basis. 

		We assume without loss of generality that \(\alpha_{n+1}\neq 0\), then 
		\[[\alpha_{n+1}f+X_{n+1},X_i]=\alpha_{n+1}AX_i,\]
		so either \(AX_i=0\) or \(AX_i\) is a nonzero multiple of some \(X_k\) for \(k\in\{1,\dots,n\}\). In the second case, there cannot exist a \(j\neq i\) such that \(AX_j\) is a nonzero multiple of \(X_k\) because we assumed that the basis was nice. So if we express \(A\) as a matrix with respect to the basis \(\{X_1, \ldots, X_n\}\), it follows that \(A\) has in each row and each column at most one nonzero element. This proves the statement if \(A^2\neq 0\).
		
		Now suppose that \(A^2=0\) and \(\dim(\im(A))\geq 2\). To simplify notation, we set
		\[Y_i:=\alpha_i f+X_i.\]
		After reordering the basis elements, we can take	\[\ker(A)=Z(\R f\ltimes_A\R^n)=\Span\{Y_{t+1},\dots,Y_{n+1}\}.\]
		Note that \(t\geq 3\) because \(\dim(\im(A))\geq 2\). 
		Assume without loss of generality that \(\alpha_1\neq 0\) and \([Y_1,Y_2]\neq 0.\) Now pick an \(i\in\{3,\dots,t\}\). Then note that
		\[[Y_1,Y_2]=A(\alpha_1X_2-\alpha_2X_1),\quad[Y_1,Y_i]=A(\alpha_1X_i-\alpha_iX_1),\quad [Y_2,Y_i]=A(\alpha_2X_i-\alpha_iX_2).\]
		So we find that
		\[\alpha_i[Y_1,Y_2]=\alpha_2[Y_1,Y_i]-\alpha_1[Y_2,Y_i].\]
		Since \(Y_1,Y_2\) and \(Y_i\) are nice basis elements and \([Y_1,Y_2]\neq 0\), the above equality implies that \(\alpha_i = 0\). In particular \(Y_3 = X_3 \notin Z(\R f\ltimes_A\R^n)=\ker(A)\), and thus \(\alpha_2=0\) because otherwise \[[Y_1,Y_3]=\frac{\alpha_1}{\alpha_2}[Y_2,Y_3] \neq 0.\]
		Hence \(\alpha_2=0\) and in the same way as in the previous case we obtain the desired result.
		
		If \(A^2=0\) and \(\dim(\im(A))=1,\) it must hold that \(\R f\ltimes_A \R^n\) is isomorphic to \(\h_3\oplus\R^n\). Note that this Lie algebra has a nice basis of the desired form. We know by \Cref{thm. |g|=|g+Rn|} that there is exactly one nice basis up to equivalence, hence the statement follows.
	\end{proof}
	This immediately implies the following corollary.
	\begin{corollary}\label[corollary]{cor. nonzero-elt-col-row}
		Let \(A\in\End(\R^{n})\), then \(\g=\R f\ltimes_A\R^n\) admits a nice basis if and only if \(\R^n\) has a basis with respect to which the matrix of \(A\) has in each row and each column at most one nonzero element.
	\end{corollary}
Note that \(\R f\ltimes_A\R^n\) is a nilpotent Lie algebra if and only if \(A\) is nilpotent. In the following proposition we compute the number of nice bases for nilpotent almost abelian Lie algebras.
\begin{proposition}\label[proposition]{prop. A-nilpotent->1 NB}
A nilpotent almost abelian Lie algebra \(\g=\R f\ltimes_A\R^n\) has exactly one nice basis up to equivalence.
\end{proposition}
\begin{proof}
	As \(A\) is nilpotent, we can find a basis \(\{X_1,\dots,X_n\}\) for \(\R^n\) such that the matrix of \(A\) with respect to this basis is its Jordan normal form. Note that the Jordan normal form of a nilpotent matrix has at most one nonzero (real) entry in every column and every row, hence it follows from \Cref{cor. nonzero-elt-col-row} that \(\{f,X_1,\dots,X_n\}\) is a nice basis for \(\g\). We will show that any other nice basis \(\mathcal{B}\) for \(\g\) is equivalent to this basis.
	
	We know by \Cref{prop. form of nice basis RfxR^n} that any nice basis for \(\g\) is equivalent to a basis \(\{f,Y_1,\dots, Y_n\}\) such that
	with respect to \(\{Y_1,\dots,Y_n\}\) the matrix of \(A\) has at most one nonzero entry in every column and every row. We can then find a permutation \(\sigma\in\mathcal{S}_n\) such that
	\(AY_i=\alpha_iY_{\sigma(i)}\) and such that \(\sigma\) is the composition of \(k=\dim(\ker(A))\) disjoint cycles \(\tau_1,\dots,\tau_k\). After reordering, these cycles are of the form
	\[\tau_1=(12\dots n_1),\quad \tau_2=(n_1+1\dots n_1+n_2),\quad \dots,\quad \tau_k=(n_1+\dots+n_{k-1}+1\dots n_1+\dots+n_k)\]
	such that for all \(1\leq i\leq k\) it holds that \(AY_{N_i}=0\), where \(N_0=0\) and \(N_i=N_{i-1}+n_i\). If we now let \(V_i:=\Span\{Y_{N_{i-1}+1},\dots,Y_{N_i}\}\) for \(1\leq i\leq k,\) it must hold that  
	\[(A\rvert_{V_i})^{n_i}=0\]
	for all \(1\leq i\leq k.\) By rescaling the vectors \(Y_i\) if necessary, \(A\) is given in its Jordan normal form. By uniqueness of the Jordan normal form, there exists a Lie algebra automorphism sending elements of \(\{f,Y_1,\dots, Y_n\}\) to multiples of elements of \(\{f,X_1,\dots, X_n\}.\) This proves the claim.
\end{proof}
In the following theorem we will show when a (not necessarily nilpotent) almost abelian Lie algebra admits a nice basis, i.e. we will translate \Cref{cor. nonzero-elt-col-row} to a condition which is easy to verify. We first introduce some notation.
\begin{notation}
For a given monic polynomial \(p(x)=x^n+a_{n-1}x^{n-1}+\dots+a_1x+a_0\), we call the \((n\times n)\)-matrix
	\[C(p(x))=
\begin{pmatrix}				
				0 & 0 & \cdots & 0 & -a_0 \\
				1 & 0 & \cdots & 0 & -a_1 \\
				0 & 1 & \cdots & 0 & -a_2 \\
				\vdots & \vdots&\ddots & \vdots & \vdots \\
				0 & 0 &\cdots & 1 & -a_{n-1}
			\end{pmatrix},\]
	the corresponding \textit{Frobenius companion matrix}. It can be shown that \(p(x)\) is the characteristic polynomial of \(C(p(x)).\)
\end{notation}

\begin{theorem}\label{thm. total-nice-basis-semidirect}
	If \(A\in\End(\R^n)\), then \(\R f\ltimes_A\R^n\) admits a nice basis if and only if up to reordering the Jordan blocks, the Jordan normal form of \(A\) is given by the block diagonal matrix
	\[\Dg(N,A_1,\dots,A_t),\]
	where \(N\) is nilpotent and each matrix \(A_i\in\C^{n_i\times n_i}\) is diagonal whose eigenvalues are the $n_i$ distinct roots of \(x^{n_i}-r_i\) for some \(r_i\in\R_0\).
\end{theorem}
\begin{proof}
	We start by proving the \textit{`if'-part} of the statement. By \Cref{cor. nonzero-elt-col-row} and its proof it suffices to show that we can find a basis for \(\R^{n_i\times n_i}\) with respect to which \(A_i\) has at most one nonzero \textit{real} entry in every column and every row. 
	
	For this let \(z_1,\dots,z_{n_i}\) be the distinct roots of \(x^{n_i}-r_i\), i.e. the spectrum of \(A_i\). We construct a set of eigenvectors \(\{w_1,\dots,w_{n_i}\}\) in the following way. If \(z_j\in\R\), then \(w_j\) is taken a real eigenvector corresponding to the eigenvalue \(z_j\). If \(z_j\in\C\setminus\R\), then \(w_j\) is an eigenvector corresponding to the eigenvalue \(z_j\), and there exists a \(k\neq j\) such that \(z_j=\overline{z_k}\) and \(w_j=\overline{w_k}.\) It now holds that the vector \(w=w_1+w_2+\dots+w_{n_i}\in\R^{n_i}\) and we also have that
	\begin{align*}
		&A_iw=z_1w_1+z_2w_2+\dots+z_{n_i}w_{n_i}\in\R^{n_i}\\
		&\quad\vdots\\
		&A_i^{n_i-1}w=z_1^{n_i-1}w_1+z_{2}^{n_i-1}w_2+\dots+z_{n_i}^{n_i-1}w_{n_i}\in\R^{n_i}.
	\end{align*}
The formula for the determinant of a \textit{Vandermonde matrix} implies that
\(\{w,A_iw,A_i^2w,\dots,A_i^{n_i-1}w\}\) is a linearly independent set of vectors in \(\mathbb{C}^n.\) Hence, \(\{w,A_iw,A_i^2w,\dots,A_i^{n-1}w\}\) forms a basis for \(\R^n.\) It is now easy to see that the matrix of \(A_i\) with respect to this basis is equal to \(C(x^{n_i}-r_i),\) what we needed to show.
	
For the \textit{`only if'-part} suppose that \(\R f\ltimes_A\R^n\) admits a nice basis. We can then find a basis \(\{X_1,\dots,X_n\}\) for \(\R^n\) such that \(A\) has in each row and each column at most one nonzero entry. More specifically we can associate to \(A\) an element \(\sigma\in\mathcal{S}_n\) such that
\(AX_i=\alpha_iX_{\sigma(i)}\) where \(\alpha_i\in\R.\) Now write \(\sigma\) as a composition of disjoint cycles \(\tau_1,\dots,\tau_k\). We can then reorder our basis \(\{X_1,\dots,X_n\}\) such that \[\tau_1=(12\dots n_1),\quad\tau_2=(n_1+1\dots n_1+n_2),\quad\dots, \tau_k=(n_1+\dots+n_{k-1}+1\dots n_1+\dots+n_k).\] Define inductively \(N_0=0\) and \(N_i=N_{i-1}+n_i\). If we let \[V_i:=\Span\{X_{N_{i-1}+1},\dots,X_{N_i}\},\quad r_i:=\alpha_{N_{i-1}+1}\dots\alpha_{N_i}\] for \(1\leq i\leq k,\) then it follows that
\[\left(A\rvert_{V_i}\right)^{n_i}=r_i \Id.\]
It now suffices to show that if \(r_i\neq 0\), then the eigenvalues of \(A\rvert_{V_i}\) correspond to the $n_i$ distinct roots of \(x^{n_i}-r_i.\) For this define
\[Y_1:=X_{N_{i-1}+1}, \quad Y_l:=\alpha_{N_{i-1}+1}\dots\alpha_{N_{i-1}+(l-1)}X_{N_{i-1}+l},\]
for \(l\in\{2,\dots, n_i\}.\) It now follows that the matrix of  \(A\rvert_{V_i}\) with respect to \(\{Y_1,\dots,Y_{n_i}\}\) is given by
\[C(x^{n_i}-r_i).\]
Hence the spectrum of \(A\rvert_{V_i}\) is equal to the roots of \(x^{n_i}-r_i\), which is exactly what we needed to show.
\end{proof}
	\begin{remark}\label[remark]{rem. char-poly} If \(\R f\ltimes_A\R^n\) admits a nice basis \(\mathcal{B}\), we can conclude from the proof of \Cref{thm. total-nice-basis-semidirect} and \Cref{prop. A-nilpotent->1 NB} that this basis must be equivalent to a basis of the form
		\[\{f,X_1,\dots,X_n\},\]
	such that the matrix of \(A\) with respect to \(\{X_1,\dots,X_n\}\) is given by a block diagonal matrix of the form
	\[\Dg\left(C\left(x^{m_1}\right),\dots,C\left(x^{m_k}\right),C\left(x^{n_1}-r_1\right),\dots,C\left(x^{n_t}-r_t\right)\right),\]
	where \(k=\dim(\ker(A))\) and \(r_1,\dots,r_t\in\R_0\). Hence, the characteristic polynomial of \(A\) is
			\[\varphi_A(x)=x^{m_1+\dots+m_k}(x^{n_1}-r_1)\dots(x^{n_t}-r_t).\]
		\end{remark}
In what follows we will generalize the statement of \Cref{prop. A-nilpotent->1 NB} to all almost abelian Lie algebras. For this we will use the following terminology.
\begin{terminology}
	Let \(p(x)\in\R[x]\) be a polynomial which does not have zero as a root. Suppose that
	\begin{equation}\label{eq. fact}
		p(x)=(x^{n_1}-r_1)\dots(x^{n_t}-r_t)=(x^{m_1}-s_1)\dots(x^{m_{t'}}-s_{t'}),\end{equation}
	where \(r_i,s_i\in\R\). We say that the two factorizations above are \textit{equivalent over \( \R\)} if \(t=t'\), after possibly reordering \((n_1,\dots,n_t)=(m_1,\dots,m_{t'})\) and there exists an \(\eta\in\R\) such that
	\[\eta^{n_i}=\frac{r_i}{s_i},\]
	for all \(i\in\{1,\dots,t\}.\)
\end{terminology} 
Using the terminology above we have the following result.
	\begin{theorem}\label[theorem]{thm. number-NB-semidirect}
	Suppose that \(\g=\R f\ltimes_A \R^n\) admits a nice basis. Let \(\alpha\in\N\) be the multiplicity of the root \(0\) in \(\varphi_A(x).\) The number of non-equivalent ways (over \(\R\)) in which we can write
	\[\frac{\varphi_A(x)}{x^\alpha}\] 
	as a product of factors of the form \(x^n-r\) for \(n\in\N\) and \(r\in\R_0\) is equal to the number of distinct nice bases there are for \(\g\). 
\end{theorem}
\begin{proof}
	If \(\varphi_A(x)=x^\alpha\) the result follows from \Cref{prop. A-nilpotent->1 NB}. We will therefore suppose that \(\varphi_A(x)\) has a nonzero eigenvalue. Hence it follows that \(\mathfrak{nil}(\g)=\R^n.\)

If \(\g\) has a nice basis \(\mathcal{B}\), we can uniquely decompose \(\R^n=U_1\oplus U_2\) where \(A\rvert_{U_1}\) is nilpotent and \(A\rvert_{U_2}\) is diagonalizable over \(\C\) without zero as an eigenvalue. By \Cref{rem. char-poly} we can suppose that the subspaces \(U_1\) and \(U_2\) are spanned by subsets of the nice basis \(\mathcal{B}\). We will denote with \(\mathcal{B}_2\) the subset of \(\mathcal{B}\) obtained by removing the elements of \(U_1\). Note that \(\mathcal{B}_2\) is a nice basis for
\[\R f\ltimes_{A\rvert_{U_2}}U_2.\]
Similarly we define \(\mathcal{B}_1.\) We will first argue that two nice bases \(\mathcal{B}\) and \(\mathcal{B}'\) for \(\g\) are equivalent if and only if \(\mathcal{B}_2\) and \(\mathcal{B}'_2\) are equivalent nice bases for \(\R f\ltimes_{A\rvert_{U_2}}U_2.\)

First suppose that two nice bases \(\mathcal{B}\) and \(\mathcal{B}'\) are equivalent. By \Cref{prop. form of nice basis RfxR^n} we can suppose that \(\mathcal{B}=\{f,X_1,\dots,X_n\}\) and \(\mathcal{B}'=\{f,Y_1,\dots,Y_n\}\). By the assumption, there exists a \(\psi\in\Aut(\g)\) mapping elements of \(\mathcal{B}\) to multiples of elements of \(\mathcal{B}'\). Note that \(\psi\) maps \(\mathfrak{nil(\g)}=\R^n\) to itself, hence we can find an \(\eta\in\R_0\) such that \(\psi(f)=\eta f.\) Note that there exists a \(k\in\N\) such that \[\ad_f^k(U_1)=0.\] 
By applying \(\psi\) to the equality above, we conclude that \(U_1\) and \(U_2\) must be \(\psi\)-invariant subspaces. Restricting \(\psi\) to \(\R f\oplus U_2\), we find that \(\mathcal{B}_2\) and \(\mathcal{B}_2'\) are equivalent nice bases for \(\R f\ltimes_{A\rvert_{U_2}}U_2\).

Conversely, suppose that \(\psi_2\) is an equivalence of the nice bases \(\mathcal{B}_2\) and \(\mathcal{B}_2'.\) By a similar argument as before we know that \(\psi_2(f)=\eta f,\) for some \(\eta\in\R_0\). By \Cref{prop. A-nilpotent->1 NB}, we can find \(\psi_1\in\Aut(\R f\ltimes_{A\rvert_{U_1}}U_1)\) mapping elements from \(\mathcal{B}_1\) to multiples of elements of \(\mathcal{B}'_1\). Moreover, we can suppose without loss of generality that \(\psi_1\) maps \(f\) to \(\eta f\). It is now clear that we can combine \(\psi_1\) and \(\psi_2\) to obtain an equivalence \(\psi\in\Aut(\g)\) for \(\mathcal{B}\) and \(\mathcal{B}'\).

By the analysis above it suffices to prove the theorem for \(\alpha=0.\) Let again \(\mathcal{B}=\{f,X_1,\dots,X_n\}\) and \(\mathcal{B}'=\{f,Y_1,\dots,Y_n\}\) be two nice bases for \(\g\). We can then suppose by \Cref{rem. char-poly} that the matrix of \(A\) with respect to \(\{X_1,\dots,X_n\}\) is given by
 \[\Dg\left(C\left(x^{n_1}-r_1\right),\dots,C\left(x^{n_t}-r_t\right)\right).\]
 Similarly, the matrix of \(A\) with respect to \(\{Y_1,\dots,Y_n\}\) is given by 
  \[\Dg\left(C\left(x^{m_1}-s_1\right),\dots,C\left(x^{m_{t'}}-s_{t'}\right)\right).\]
 In particular we know that \(A\) permutes the vectors \(\{X_1,\dots,X_n\}\) by a permutation \(\sigma\) whose cycle structure equals \((n_1,n_2,\dots,n_t)\). On the other hand \(A\) permutes the vectors \(\{Y_1,\dots,Y_n\}\) by a permutation \(\sigma ',\) whose cycle structure equals \((m_1,m_2,\dots,m_{t'})\).
 
Note that if the condition from the theorem is satisfied, we can construct a Lie algebra automorphism sending elements of \(\mathcal{B}\) to multiples of elements of \(\mathcal{B}'.\) To prove the converse we will start by showing that if \(\mathcal{B}\) and \(\mathcal{B}'\) are equivalent, then the \(\sigma\) and \(\sigma'\) have the same cycle structure. For this suppose that \(\psi\in\Aut(\g)\) maps elements from \(\mathcal{B}\) to multiples of elements of \(\mathcal{B}'.\) As \(\mathfrak{nil}(\g)=\R^n\) is mapped to itself by \(\psi\), we can find an \(\varepsilon\in\mathcal{S}_n\) such that
 \[\psi(X_i)=\eta_iY_{\varepsilon(i)}.\]
 As mentioned earlier, this also implies that \(\psi(f)=\eta f\) for some \(\eta\in\R_0\). Note that for all \(i\in\{1,\dots,n\}\) it holds that
 \[\psi([f,X_i])\in\Span\{Y_{\varepsilon(\sigma(i))}\}\setminus\{0\}, \quad [\psi(f),\psi(X_i)]\in\Span\{Y_{\sigma'(\varepsilon(i))}\}\setminus\{0\}. \]
 As \(\psi\in\Aut(\g),\) we conclude that \(\sigma=\varepsilon^{-1}\circ\sigma'\circ\varepsilon.\) Hence, after possibly reordering, we have that \(\varepsilon=e\in\mathcal{S}_n\). In particular \(t=t'\) and \((n_1,\dots,n_t)=(m_1,\dots,m_{t'}).\)
 
To end the proof, fix a disjoint cycle \(\tau_i=(a_1,\dots,a_{n_i})\) and define \[V:=\Span\{X_{a_1},\dots,X_{a_{n_i}}\},\quad W:=\Span\{Y_{a_1},\dots,Y_{a_{n_i}}\}.\]
 Note that \(\psi(V)=W\) and for all \(j\in\{1,\dots,n_i\}\) we have
 \[(\psi\circ A\rvert_{V}\circ\psi^{-1})(Y_{a_j})=\psi([f,\psi^{-1}(Y_{a_j})])=[\eta f,Y_{a_j}]=\eta A\rvert_{W}(Y_{a_j}).\]
 The above implies that \(C(x^{n_i}-r_i)\) must be conjugate to \(\eta C(x^{n_i}-s_i).\) Hence their characteristic polynomials must be the same. Note that the characteristic polynomial of \(\eta C(x^{n_i}-s_i)\) is given by
 \[x^{n_i}-\eta^{n_i} s_i.\]
Hence we conclude that
 \[\eta^{n_i}=\frac{r_i}{s_i}.\]
As the cycle \(\tau_i\) was chosen arbitrary, the above holds for all \(i\in\{1,\dots,t\}\), concluding the proof.
\end{proof}
In what follows we illustrate the results obtained in this section with several examples. Moreover we will derive some interesting consequences from these results, which were already mentioned in the introduction.
	\begin{example}
		Consider the matrix
		\[A=\begin{pmatrix}
			-2&3\\
			-4&-2
		\end{pmatrix}.\]
		Its eigenvalues are \(-2+2\sqrt{3}i\) and \(-2-2\sqrt{3}i\). It follows from \Cref{thm. total-nice-basis-semidirect} that the Lie algebra \(\R f\ltimes_A\R^2\) does not admit a nice basis. 
		Now consider 
		\[B=\begin{pmatrix}
			A&0\\
			0&4
		\end{pmatrix}\in\R^{3\times 3},\]
		The eigenvalues of \(B\) are \(4,-2+2\sqrt{3}i\) and \(-2-2\sqrt{3}i\). It holds that these eigenvalues are the three third roots of \(64.\) We thus find from \Cref{thm. total-nice-basis-semidirect} that \(\R f\ltimes_B\R^3\) admits a nice basis. Indeed, we can construct this basis explicitly by using the method used in the proof of \Cref{thm. total-nice-basis-semidirect}. The eigenvectors are given by
		\[w_1=\begin{pmatrix}
			\sqrt{3}/2 i\\
			1\\
			0
		\end{pmatrix},\quad w_2=\begin{pmatrix}
			-\sqrt{3}/2 i\\
			1\\
			0
		\end{pmatrix},\quad \quad w_3=\begin{pmatrix}
			0\\
			0\\
			1
		\end{pmatrix}.\]
		We now find that
		\[X_1=w_1+w_2+w_3=\begin{pmatrix}
			0\\2\\1
		\end{pmatrix},\quad X_2=BX_1=\begin{pmatrix}
			6\\-4\\4
		\end{pmatrix},\quad X_3=BX_2=B^2X_1=\begin{pmatrix}
			-24\\-16\\16
		\end{pmatrix},\]
		forms a basis for \(\R^3\). As \(B^3= 2^6\Id\), it follows that \(\{f,X_1,X_2,X_3\}\) forms a nice basis for  \(\R f\ltimes_B\R^3.\) As \(\varphi_B(x)=x^3-64\), we know by \Cref{thm. number-NB-semidirect} that this is the only nice basis up to equivalence.
	\end{example}
The proofs of \Cref{thm. total-nice-basis-semidirect} and \Cref{thm. number-NB-semidirect} also hold if we replace every occurrence of \(\R\) by a field \(K\subseteq \C\). In particular, if \(K=\C\), \Cref{thm. total-nice-basis-semidirect} leads to the following.
	\begin{proposition}
		If \(M\in\C^{n\times n}\), then \(\C f\ltimes_M\C^n\) admits a nice basis if and only if \(M\) is the sum of a semisimple matrix \(S\) and a nilpotent matrix \(N\) such that \(SN=NS=0\).
	\end{proposition}

Let \(K\subseteq L\) be a field extension and let \(\g^K\) be a Lie algebra over \(K\). By extending the scalars, we find a new Lie algebra \(\g^L = \g^K\otimes_K L\). If \(\g^K\) has a nice basis \(\{X_1,\dots, X_n\}\), then \(\{X_1\otimes 1,\dots,X_n\otimes 1\}\) is a nice basis for \(\g^K\otimes_K L\). However, it could be the case that non-equivalent nice bases become equivalent over the larger field. This raises the following question.
\begin{question}
	If $K \subseteq L$ is a field extension, and $\g^K$ a Lie algebra over $K$, is there a relation between
	\(\nu_{\g^K}\) and \(\nu_{\g^L}\)?
\end{question}

So far, there is no example where the numbers in the question above are unequal. As the decompositions of \Cref{thm. total-nice-basis-semidirect} depend on the field over which we work, it allows us to present the first examples for the extensions \(\R\subseteq \C\) and \(\Q\subseteq\R\).
\begin{example}
	Consider the matrix
	\[A=\begin{pmatrix}
		1&-1\\
		1&1
	\end{pmatrix}.\]
	Note that \(A\) is diagonalizable over \(\C\) as \(\Spec(A)=\{1+i,1-i\}\). Since \((1+i)^2=2i=-(1-i)^2,\) we find by \Cref{thm. total-nice-basis-semidirect} that \(\R f\ltimes_A\R^2\) does not admit a nice basis. However, note that
	\[\C f\ltimes_A \C^2\cong (\R f\ltimes_A\R^2)\otimes_\R \C\]
	and since \(A\) is diagonalizable over \(\C\) the Lie algebra \(\C f\ltimes_A \C^2\) admits a nice basis.
	
Similarly, for
		\[B=\begin{pmatrix}
			2&1\\
			1&1
		\end{pmatrix},\]  it holds that \(\R f\ltimes_B\R^2\) admits a nice basis but \(\mathbb{Q} f\ltimes_B\mathbb{Q}^2\) does not.
\end{example}
The example above leads to the following corollary.
\begin{corollary}\label[corollary]{cor. NB-not-preserved-under-fld-ext.}
Passing to a subfield need not preserve the existence of a nice basis.
\end{corollary}
We already know that the number of nice bases for a decomposable Lie algebra can be arbitrarily large. By using \Cref{thm. number-NB-semidirect} we have the following result.
\begin{corollary}
	For every \(n\in\N\) and any field \(K\subseteq\C \) there exists an indecomposable Lie algebra over \(K\) with exactly \(n\) nice bases up to equivalence.
\end{corollary}
\begin{proof}
	Note that the \(3\)-dimensional Heisenberg Lie algebra has one nice basis up to equivalence over any field \(K\subseteq \C\).
	
	Now suppose \(n\geq 2\) and let \(A_n\) be the \(2(n-1)\times 2(n-1)\) matrix
	\[\Dg\left(C(x^2-1),\dots, C(x^2-1)\right).\]
	Note that
	\[\varphi_{A_n}(x)=(x-1)^{n-1}\cdot(x+1)^{n-1}.\]
	 It now follows that \(\varphi_{A_n}(x)\) has exactly \(n\) non-equivalent factorizations over \(K\), so we can conclude that the Lie algebra
	\[K f\ltimes_{A_n} K^{2(n-1)}\]
	has exactly \(n\) nice bases up to equivalence. Moreover, this Lie algebra is clearly indecomposable.
\end{proof}
We will end this section with an example where two factorizations with the same number of factors are not equivalent.
	\begin{example}
		Consider the matrix \(A\in\R^{4\times 4}\) given by
		\[A=\begin{pmatrix}
			1&0&0&0\\
			0&-1&0&0\\
			0&0&2&0\\
			0&0&0&-2
		\end{pmatrix}.\]
		It is clear that \(\R f\ltimes_A\R^4\) has at least one nice basis. Note that the non-equivalent factorizations of the characteristic polynomial are given by
		\begin{itemize}
			\item \((x-2)(x+2)(x-1)(x+1),\)
			\item \((x^2-4)(x-1)(x+1),\)
			\item \((x^2-1)(x-2)(x+2),\)
			\item \((x^2-1)(x^2-4).\)
		\end{itemize}
		So we find that \(\R f\ltimes_A\R^4\) has 4 distinct nice bases. Note that the second and third factorization are not equivalent over \(\R\) because there does not exist an \(\eta\in\R\) such that
		\(\eta^2=4\text{ and }\eta=1/2.\)
		 If we denote by \(\{v_1,v_2,v_3,v_4\}\) the standard basis for \(\R^4\), then the second and third factorization correspond to the nice bases
		\[\{f,v_1,v_2,v_3+v_4,v_3-v_4\}\text{ and } \{f,v_1+v_2,v_1-v_2,v_3,v_4\}\]
		respectively, which are not equivalent.
	\end{example}
	\section{Application: nice bases for 3-dimensional Lie algebras}\label{sect. 3-dim}
	This section applies the results of the previous section to determine the number of nice bases of \(3\)-dimensional Lie algebras. Below we use notations from a classification for all \(3\)-dimensional solvable Lie algebras in \cite{DeGraaf}, implying in particular that they are almost abelian. 
		
		The following well-known result describes when two almost abelian Lie algebras are isomorphic, but we give the proof for completeness.
		\begin{lemma}
			The Lie algebras \(\g:=\R f\ltimes_A\R^n\) and \(\h:=\R g\ltimes_B\R^n\) are isomorphic if and only if there exist \(Q\in \GL(n,\R)\) and \(c\in\R_0\) such that
			\[A=cQ^{-1}BQ.\]
		\end{lemma}
\begin{proof}		
		It is clear that the condition on the matrices is sufficient for \(\g\) and \(\h\) to be isomorphic. For the converse, first assume that \(\g\) and \(\h\) nilpotent, so \(A\) and \(B\) are nilpotent matrices. As 
			\[\dim(\im(A^{i-1}))=\dim(\gamma_i(\g))=\dim(\gamma_i(\h))=\dim(\im(B^{i-1})),\]
			for all \(i\in\N_{\geq 2}\), the matrices \(A\) and \(B\) have the same Jordan blocks up to permutation, therefore \(A\) and \(B\) must be conjugate. 
			
			For the other case, assume that \(A\) and \(B\) are not nilpotent, then
			\[\nil(\g)=\R^{n},\qquad \nil(\h)=\R^{n}\]
		and thus any isomorphism \(\varphi:\g\rightarrow \h\) satisfies \(\varphi\left(\R^n\right) = \R^n\). Hence \(\varphi(f)=c g+ Y\) for some \(c\in\R_0\) and \(Y\in\R^n\). Moreover, we can express \(\varphi\rvert_{\R^n}\) as a matrix \(Q\) with respect to the canonical basis. The result now follows as for all \(X\in\R^n\) 
			\[QAX=\varphi([f,X])=[c g+Y,QX]=c B QX.\]
		\end{proof}
		For two matrices \(A,B\in\R^{n\times n}\) the relation
		\[A\sim B\iff \exists c\in\R_0, Q\in \GL(n,\R) : A=cQ^{-1}BQ\]
		defines an equivalence relation on \(\R^{n\times n}\). For \(n=2\) we have the following representatives for the different equivalence classes.
		\[A_\lambda=\begin{pmatrix}
			1&0\\
			0&\lambda
		\end{pmatrix},\quad B=\begin{pmatrix}
			0&0\\
			0&0
		\end{pmatrix},\quad C=\begin{pmatrix}
			0&1\\
			0&0
		\end{pmatrix},\quad D=\begin{pmatrix}
			1&1\\
			0&1
		\end{pmatrix},\quad E_{\mu}=\begin{pmatrix}
			\mu&1\\
			-1&\mu\end{pmatrix},\]
		where \(\lambda\in[-1,1]\) and \(\mu\in\R^+\). From \cite{Khabirov} it follows that the remaining \(3\)-dimensional Lie algebras are the simple ones, namely \(\mathfrak{sl}_2(\R)\) and \(\mathfrak{so}_3(\R)\). For \(\mathfrak{sl}_2(\R)\) we will use the basis
		\[e_1=\begin{pmatrix}
			1&0\\
			0&-1
		\end{pmatrix},\quad e_2=\begin{pmatrix}
			0&1\\
			0&0
		\end{pmatrix},\quad 
		e_3=\begin{pmatrix}
			0&0\\
			1&0
		\end{pmatrix}.\]
		For \(\mathfrak{so}_3(\R)\) we will use the basis
		\[f_1=\begin{pmatrix}
			0&0&0\\
			0&0&1\\
			0&-1&0
		\end{pmatrix},\quad f_2=\begin{pmatrix}
			0&0&1\\
			0&0&0\\
			-1&0&0
		\end{pmatrix},\quad 
		f_3=\begin{pmatrix}
			0&1&0\\
			-1&0&0\\
			0&0&0
		\end{pmatrix}.\] In the table below we summarize the results that we will obtain in  this section. 
	 We will use the notation \(v_1=(1,0)^\top\) and \(v_2=(0,1)^\top\).
		\begin{center}
			\renewcommand{\arraystretch}{1.3}
			\begin{tabular}{c|c|c}
				
				Lie algebra \(\g\) & non-equivalent nice bases& \(\nu_\g\) \\
				\hline
				\hline
					\(\R f\ltimes_{A_{-1}}\R^2\)&\makecell[c]{\(\{f,v_1,v_2\}\)\\ \(\{f,v_1+v_2,v_1-v_2\}\)}&2\\
				\hline
				\(\R f\ltimes_{A_{\lambda}}\R^2\quad(\lambda\in]-1,1])\)&\(\{f,v_1,v_2\}\)&1\\
				\hline
				\(\R^3\cong \R f\ltimes_{B}\R^2\)&\(\{f,v_1,v_2\}\)&1 \\
				\hline
				\(\h_3\cong\R f\ltimes_C\R^2\)&\(\{f,v_1,v_2\}\)&\(1\) \\
				\hline
							\(\R f\ltimes_{D}\R^2\)&none&0\\
				\hline
				\(\R f\ltimes_{E_0}\R^2\)&\(\{f,v_1,v_2\}\)&1\\
				\hline
				\(\R f\ltimes_{E_\mu}\R^2 \quad(\mu>0)\)&none&0\\
				\hline
				\(\mathfrak{sl}_2(\R)\)&\makecell[c]{\(\{e_1,e_2,e_3\}\)\\ \(\{e_1,e_2+e_3,e_2-e_3\}\)}&2\\
				\hline
				\(\mathfrak{so}_3(\R)\)&\(\{f_1,f_2,f_3\}\)&1\\
			\end{tabular}
		\end{center}
		
 As \(B\) and \(C\) are nilpotent, both \(\R f\ltimes_{B}\R^2\) and \(\R f\ltimes_{C}\R^2\) have exactly one nice basis up to equivalence. The characteristic polynomial of \(A_\lambda\) is equal to
	\[\varphi_{A_\lambda}(x)=(x-1)(x-\lambda)\]
and thus the Lie algebra \(\R f\ltimes_{A_\lambda}\R^2\) has exactly one nice basis if \(\lambda\neq-1\) and two nice bases if \(\lambda = -1\). As the matrix
	\[D^2=\begin{pmatrix}
		1&2\\
		0&1
	\end{pmatrix}\]
	is not diagonalizable, the Lie algebra \(\R f\ltimes_{D}\R^2\) does not admit a nice basis. The eigenvalues of \(E_\mu\) for \(\mu\in\R^+\) are given by \(\mu+i\) and \(\mu-i\). Hence \(\R f\ltimes_{E_\mu}\R^n\) admits a nice basis if and only if
	\[(\mu+i)^2\in\R\iff (\mu^2-1)+2\mu i\in\R\iff \mu=0.\]
	In the case \(\mu = 0\) the characteristic polynomial is equal to
	\[\varphi_{E_0}(x)=x^2+1,\]
	and therefore \(\R f\ltimes_{E_0}\R^n\) has one nice basis up to equivalence.

Finally, we consider the simple Lie algebras \(\mathfrak{sl}_2(\R)\) and \(\mathfrak{so}_3(\R)\). Both admit a nice basis, as the standard basis mentioned before is already nice. Indeed, for \(\mathfrak{sl}_2(\R)\) with basis \(\{e_1,e_2,e_3\}\) we have
	\[ [e_1,e_2]=2e_2,\quad [e_1,e_3]=-2e_3,\quad[e_2,e_3]=e_1.\]
	On the other hand, for \(\mathfrak{so}_3(\R)\) with basis \(\{f_1,f_2,f_3\}\) we have
	 \[[f_2,f_3]=f_1,\quad[f_3,f_1]=f_2,\quad[f_1,f_2]=f_3.\] 
 Every semisimple Lie algebra is unimodular and its derivations are inner, hence the pre-Einstein derivation is \(0\). In particular, it does not provide any information about the existence or the number of nice bases. Consequently, a different technique is needed to count the number of nice bases for the \(3\)-dimensional simple Lie algebras.
	\begin{proposition}
The Lie algebra \(\mathfrak{sl}_2(\R)\) has two nice bases up to equivalence and \(\mathfrak{so}_3(\R)\) has one nice basis up to equivalence.
	\end{proposition}
\begin{proof}
	Note that any \(3\)-dimensional simple Lie algebra \(\g\) is isomorphic to either \(\mathfrak{sl}_2(\R)\) or to \(\mathfrak{so}_3(\R).\) Moreover, the Killing form of \(\mathfrak{so}_3(\R)\) is negative definite, whereas it is not for \(\mathfrak{sl}_2(\R)\).
	
	Suppose \(\g\) is a \(3\)-dimensional simple Lie algebra and let \(\{X_1,X_2,X_3\}\) be a nice basis. It then holds that 
\[[X_2,X_3]=\alpha X_{\sigma(1)},\quad[X_3,X_1]=\beta X_{\sigma(2)},\quad [X_1,X_2]=\gamma X_{\sigma(3)}\]
for some \(\sigma\in\mathcal{S}_3\) and \(\alpha,\beta,\gamma\in\R_0.\) From the Jacobi identity it immediately follows that \(\sigma\) does not have order \(3\), as otherwise a linear combination of \(X_1,X_2\) and \(X_3\) with nonzero coefficients should equal zero, which is impossible. 

Therefore \(\sigma\) has order \(1\) or \(2\), and after permuting the basis vectors we thus find two possibilities for \(\sigma\), namely corresponding to \(\sigma = e\) and \(\sigma = (12)\), i.e.~:
\begin{enumerate}[(1)]
	\item\([X_2,X_3]=\alpha X_{1},\quad[X_3,X_1]=\beta X_{2},\quad[X_1,X_2]=\gamma X_{3}\)
	\item\([X_2,X_3]=\alpha X_{2},\quad[X_3,X_1]=\alpha X_{1},\quad[X_1,X_2]=\gamma X_{3}\)
\end{enumerate}
Note that we can find in the last case a Lie algebra isomorphism sending vectors from \(\{X_1,X_2,X_3\}\) to multiples of vectors from \(\{e_1,e_2,e_3\}\) which is a nice basis for \(\mathfrak{sl}_2(\R)\). We thus only need to consider the first case. For this, first note that the eigenvalues of Killing form are given by
\[-2\beta\gamma,\quad -2\alpha\gamma,\quad -2\alpha\beta.\]
From this we can conclude that the Lie algebra with basis \(\{X_1,X_2,X_3\}\) and Lie bracket   
\[[X_2,X_3]=\alpha X_{1},\quad[X_3,X_1]=\beta X_{2},\quad[X_1,X_2]=\gamma X_{3}\]
is isomorphic to \(\mathfrak{so}_3(\R)\) if and only if \begin{equation}\label{eq. cond. sgn}
\sgn(\alpha)=\sgn(\beta)=\sgn(\gamma).\end{equation}
If this is the case, we can find a Lie algebra isomorphism sending vectors from \(\{X_1,X_2,X_3\}\) to multiples of vectors from \(\{f_1,f_2,f_3\}.\) We can thus conclude from this that \(\mathfrak{so}_3(\R)\) admits only one nice basis up to equivalence.

If condition \eqref{eq. cond. sgn} is not satisfied, we know that the Lie algebra must be isomorphic to \(\mathfrak{sl}_2(\R)\) and it is clear that \(\{X_1,X_2,X_3\}\) and \(\{e_1,e_2,e_3\}\) are not equivalent. Similarly, any other basis \(\{Y_1,Y_2,Y_3\}\) for which 
\[[Y_2,Y_3]=\alpha' Y_{1},\quad[Y_3,Y_1]=\beta' Y_{2},\quad[Y_1,Y_2]=\gamma' Y_{3}\]
is equivalent to \(\{X_1,X_2,X_3\}\) if \(\alpha',\beta'\) and \(\gamma'\) do not satisfy condition \eqref{eq. cond. sgn}. We can thus conclude from this that \(\mathfrak{sl}_2(\R)\) has two nice bases up to equivalence, more specifically the bases
\[\{e_1,e_2,e_3\},\quad \{e_1,e_2+e_3,e_2-e_3\}.\]
\end{proof}
	\section{Nice bases for nilpotent Lie algebras associated to graphs}\label{sect. graphs}
	In this section we will study nilpotent Lie algebras associated to simple graphs \(\G=(V,E)\), where \(V\) denotes the set of vertices and \(E\) the set of edges. To define these Lie algebras, start from the free \(c\)-step nilpotent Lie algebra on \(V\), say \(\n_c(V)\) and consider the ideal \(I_\G\) generated by
	\[\{[v_1,v_2]\mid v_1,v_2\in V, \{v_1,v_2\}\notin E\}.\]
	The \(c\)-step nilpotent Lie algebra associated to the graph \(\G\) is then given by
	\[\n_{\G,c}=\frac{\n_c(V)}{I_\G}.\]
The results that we will prove in this section are summarized by the following theorem.
\begin{theorem}[store=NB-graphs]\label{thm. general-NB-graph}
	Let \(\G\) be a non-empty graph and \(c > 0\) an integer. The Lie algebra \(\n_{\G,c}\) admits a nice basis if and only if one of the following conditions holds:
	\begin{enumerate}[(1)]
		\item \(c\leq2\),
		\item \(c=3\) and \(\G\) does not contain a \(3\)-cycle,
		\item \(c=4\) and \(\G\) does not contain a path graph with three vertices.
	\end{enumerate}
\end{theorem}
Before we prove the result above, we need some additional results concerning nilpotent Lie algebras associated to graphs. The following result is given in \cite[p.~5230]{Dere-Witdouck23}.
	\begin{theorem}[Universal property]
		Let \(\G=(V,E)\) be a simple graph. For any \(c\)-step nilpotent Lie algebra \(\g\) and any map \(f:V\rightarrow\g\) such that \([f(v_1),f(v_2)]=0\) if \(\{v_1,v_2\}\notin E,\) there exists a unique Lie algebra morphism \(\varphi:\n_{\G,c}\rightarrow \g\) such that \(\varphi\rvert_V=f.\)
	\end{theorem}
	Note that a \(c\)-step nilpotent Lie algebra associated to a complete graph with \(d\) vertices is isomorphic to the free \(c\)-step nilpotent Lie algebra on \(d\) generators. We will denote this Lie algebra by \(\n_{d,c}\). For these specific types of nilpotent Lie algebras associated to graphs we have the following result from \cite[Theorem 3.11]{Conti-Rossi-Ad-invariant}.
	\begin{proposition}
		A free \(c\)-step nilpotent Lie algebra on \(d\) generators admits a nice basis if and only if either \(c\leq 2\) or \(c\leq 4\) and \(d=2.\)
	\end{proposition} 
	In the proof of the following lemma we will use the fact that if \(\{X_1,\dots,X_d\}\) is a set of vectors in a nilpotent Lie algebra \(\n\) whose projections onto \(\n/\gamma_2(\n)\) form a basis, then \(\{X_1,\dots,X_d\}\) is a set of generators for \(\n.\) We will also use this fact in the proofs in the following sections. 
	\begin{lemma}\label[lemma]{lem. complete->1NB}
		If the Lie algebra \(\n_{d,c}\) admits a nice basis, then it is unique up to equivalence.
	\end{lemma}
	\begin{proof}
		We can view \(\n_{d,c}\) as the \(c\)-step nilpotent Lie algebra associated to a complete graph on the vertices \(\{v_1,\dots,v_d\}.\) 
		If \(\{X_1,\dots,X_n\}\) is nice basis for \(\n_{d,c}\), we can assume after possibly reordering that the projections of
		\[X_1,\dots,X_d\]
		onto \(\n_{d,c}/\gamma_2(\n_{d,c})\) form a basis. By the universal property we can now find a Lie algebra morphism \(\varphi:\n_{d,c}\rightarrow\n_{d,c}\) sending \(v_i\) to \(X_i\) for all \(i\in\{1,\dots,d\}.\) This is an isomorphism as it maps generators to generators. From this we can conclude that all nice bases for \(\n_{d,c}\) are equivalent, as every other element in the nice basis can be written as a bracket of these generators, and thus these are mapped onto each other.
	\end{proof}
	The universal property can also be used to prove the following.
	\begin{lemma}\label[lemma]{lem. pi:G->G'}
		Let \(\G=(V,E)\) be a simple graph, let \(c>1\) and suppose that \(\G'=(V',E')\) is a subgraph of \(\G\). Then there exists a Lie algebra morphism 
		\[\pi:\n_{\G,c}\rightarrow\n_{\G',c}\]
		such that for any \(v\in V\) it holds that
		\[\pi(v)=\begin{cases}
			v\text{ if }v\in V'\\
			0 \text{ if }v\notin V'
		\end{cases}.\]
	\end{lemma}
	As noted in the introduction, the following result is already known by \cite{Lauret-Will-first-mention}.  
	\begin{lemma}
		Any two-step nilpotent Lie algebra associated to a graph \(\G=(V,E)\) admits a nice basis.
	\end{lemma}
	\begin{proof}
		Suppose that \(V=\{v_1,\dots,v_n\}\) and consider the sets
		\[B_1:=\{v_1,\dots,v_n\},\quad B_2:=\{[v_i,v_j]\mid \{i,j\}\in E \text{ and } i<j\}.\]
		It now follows immediately that \(\mathcal{B}:=B_1\cup B_2\) is a nice basis for \(\n_{\G,2}.\)
	\end{proof}
Nilpotency classes larger than \(2\) are discussed in the following sections.
	\subsection{3-step nilpotent Lie algebras associated to graphs}
	We will show in the following example that \(\n_{\G,3}\) admits a nice basis if \(\G\) does not contain a \(3\)-cycle.
	\begin{example}\label[example]{ex. NB-if-no-3-cycle}
	Suppose that \(\G=(V,E)\) does not contain a \(3\)-cycle. We will construct a nice basis for \(\n_{\G,3}\) explicitly. As explained in \cite{Wade}, every nilpotent Lie algebra associated to a graph has a Lyndon basis. For \(\n_{\G,3}\) this basis corresponds to the union of the following four sets
	\begin{itemize}
		\item\(B_1:=V=\{v_1,\dots,v_n\}\)
		\item\(B_2:=\{[v_i,v_j]\mid i<j\text{ and }\{i,j\}\in E\}\)
		\item \(B_3:=\{[v_i,[v_j,v_k]]\mid i<k\text{ and } \{v_i,v_j\},\{v_j,v_k\}\in E\}\)
		\item \(B_4:=\{[v_i,[v_i,v_j]],[v_j,[v_i,v_j]]\mid i<j \text{ and } \{v_i,v_j\}\in E\}\)
	\end{itemize}
We now claim that
 \[\mathcal{B}=B_1\cup B_2\cup B_3\cup B_4\] is indeed a nice basis. First note that the Lie bracket of any two elements in \(B_1\) is a multiple of an element in \(B_2.\) Moreover, for each element \([v_i,v_j]\in B_2\) there exists a unique pair of vectors in \(B_1\), namely \(v_i\) and \(v_j\), whose Lie bracket is a multiple of \([v_i,v_j].\) Now note that for all \(v_i\in B_1\) and \([v_j,v_k]\in B_2\) it holds that if \(i\neq j,k\)
	\[[v_i,[v_j,v_k]]=\begin{cases}
			0 &\text{ if }\{v_i,v_j\},\{v_i,v_k\}\notin E\\
			[v_i,[v_j,v_k]]&\text{ if }i<k,\{v_i,v_j\}\in E\\
			[v_k,[v_j,v_i]]&\text{ if }k<i,\{v_i,v_j\}\in E\\
			-[v_i,[v_k,v_j]]&\text{ if }i<j,\{v_i,v_k\}\in E\\
			[v_j,[v_k,v_i]]&\text{ if }j<i,\{v_i,v_k\}\in E\\
		\end{cases}.\]
		If \(i=k\), then \[[v_i,[v_j,v_i]]=\begin{cases}
			[v_i,[v_j,v_i]] &\text{ if }j<i\\
			-[v_i,[v_i,v_j]]& \text{ if }i<j
		\end{cases}\]
		If \(i=j\), then \[[v_i,[v_i,v_k]]=\begin{cases}
			[v_i,[v_i,v_k]] &\text{ if }i<k\\
			-[v_i,[v_k,v_i]]& \text{ if }k<i
		\end{cases}.\]
		From this it is clear that \([v_i,[v_j,v_k]]\) is always a multiple of an element in \(B_3\cup B_4\). Moreover, given an element \([v_i,[v_j,v_k]]\in B_3\cup B_4\) there is only one choice of a pair of elements in \(B_1\) and \(B_2\) whose Lie bracket is a multiple of \([v_i,[v_j,v_k]].\) We can thus conclude that \(\mathcal{B}\) is a nice basis for \(\n_{\G,3}.\)
	\end{example}
	We can now use the above example to prove the following.
	\begin{proposition}\label[proposition]{prop. 3-cycle}
	A \(3\)-step nilpotent Lie algebra associated to a graph \(\mathcal{G}\) admits a nice basis if and only if \(\mathcal{G}\) does not contain a \(3\)-cycle.
	\end{proposition}
	\begin{proof}
		The \textit{`if'-}part follows from \Cref{ex. NB-if-no-3-cycle}. For the \textit{`only if'-}part suppose by contradiction that \(\G\) contains a \(3\)-cycle \(\G'\) and let \(\{X_1,\dots,X_n\}\) be a nice basis for \(\n_{\G,3}\).
		
		We find from \Cref{lem. pi:G->G'} that there is a surjective Lie algebra morphism
		\[\pi:\n_{\G,3}\rightarrow\n_{\G', 3}.\]
		Let \(p:\n_{\G',3}\rightarrow \n_{\G',3}/\gamma_2(\n_{\G',3})\) denote the projection morphism. Without loss of generality we can now suppose that \(\{(p\circ\pi)(X_1),(p\circ\pi)(X_2),(p\circ\pi)(X_3)\}\) forms a basis for \(\n_{\G',3}/\gamma_2(\n_{\G',3})\).
		
		We now claim that the Lie subalgebra \(\n\) of \(\n_{\G,3}\) generated by \(X_1,X_2\) and \(X_3\) is a free \(3\)-step nilpotent Lie algebra. To show this we will use the universal property. Let \(f:\{X_1,X_2,X_3\}\rightarrow \mathfrak{h}\) be any map, where \(\mathfrak{h}\) is a \(3\)-step nilpotent Lie algebra. Since \(\pi(X_1),\pi(X_2)\) and \(\pi(X_3)\) are generators for \(\n_{\G',3}\), it follows that we can find a unique Lie algebra morphism \[\varphi:\n_{\G',3}\rightarrow \mathfrak{h}\] such that for all \(i\in\{1,2,3\}\) we have that \(\varphi(\pi(X_i))=f(X_i).\) We can now look at the composition \[\varphi\circ \pi:\n\rightarrow \mathfrak{h}.\] Uniqueness follows from the fact that \(\n\) is generated by \(\{X_1,X_2,X_3\}.\) 

		Since we assumed that \(\{X_1,\dots,X_n\}\) is a nice basis for \(\n_{\G,3}\), we conclude that \(\n\cong\n_{3,3}\) has a nice basis as well which is impossible.
	\end{proof}
	Note that the graphs without \(3\)-cycles form a large class, including for example trees, bipartite graphs, cycle graphs \(C_m\) and prism graphs \(P_n\) for \(m,n\geq 4\).
	
	\subsection{4-step nilpotent Lie algebras}
	We will show that a \(4\)-step nilpotent Lie algebra associated to a graph \(\G\) admits a nice basis if and only if \(\G\) does not have \(\G'=(V',E')\) as a subgraph where \(V'=\{v_1,v_2,v_3\}\) and \(E'=\{\{v_1,v_2\},\{v_2,v_3\}\}.\)
	\begin{center}
		\begin{tikzpicture}[
			node/.style={circle, fill=black, inner sep=2pt},
			label distance=2mm
			]
			
			\node[node, label=above:{\(v_1\)}] (v1) at (-1,0) {};
			\node[node, label=above:{\(v_2\)}] (v2) at (0,0) {};
			\node[node, label=above:{\(v_3\)}] (v3) at (1,0) {};
			
			\draw (v1) -- (v2) -- (v3);		
		\end{tikzpicture}
	\end{center}
	 A tool that we will need to prove this fact is the definition of a \textit{Carnot Lie algebra}.
	\begin{definition}
		A \(c-\)step nilpotent Lie algebra \(\n\) is a Carnot Lie algebra if it has a grading \(\n=\displaystyle \bigoplus_{i=1}^c \n_i\) such that \(\n_1\) generates \(\n.\) Here, a grading means that \([\n_i,\n_j] \subseteq \n_{i+j}\) for all \(i, j \in \{1, \ldots, c\}\).
	\end{definition}
\noindent	From the definition above it follows immediately that any nilpotent Lie algebra associated to a graph is a Carnot Lie algebra.  
	
	We can associate to a given \(c\)-step nilpotent Lie algebra \(\n\) a Carnot Lie algebra \(\Car(\n)\). As a vector space it corresponds to the direct sum 
	\[\bigoplus_{i=1}^c \frac{\gamma_i(\n)}{\gamma_{i+1}(\n)},\]
	and for \(X\in\gamma_i(\n)\) and \(Y\in \gamma_j(\n)\) we define
	\[
	[X+\gamma_{i+1}(\n),Y+\gamma_{j+1}(\n)]:=[X,Y]+\gamma_{i+j+1}(\n).\]
	This definition can now be extended linearly to the entire vector space. It is clear that if \(\n\) is a Carnot Lie algebra, then \(\Car(\n)\cong\n\).

	If \(\{X_1,\dots,X_n\}\) is a nice basis for a \(c\)-step nilpotent Lie algebra \(\n\), we can suppose without loss of generality that the first \(k_1\) vectors of this basis belong to \(\n\setminus\gamma_{2}(\n)\), the following \(k_2\) vectors belong to \(\gamma_{2}(\n)\setminus\gamma_{3}(\n)\) and so on. It then follows that
	\[\{X_1+\gamma_2(\n)\dots,X_{k_1}+\gamma_2(\n),X_{k_1+1}+\gamma_3(\n),\dots,X_{n}+\gamma_{c+1}(\n)\}\]
	forms a nice basis for \(\Car(\n).\) It thus follows that the existence of a nice basis for \(\n\) implies the existence of a nice basis for \(\Car(\n).\)

The following vectors form a Lyndon basis for \(\n_{\mathcal{G'},4}\), as discussed in \cite{Wade}:
\begin{itemize}
	\item \(v_1,v_2,v_3,\)
	\item \([v_1,v_2],[v_2,v_3],\)
	\item \([v_1,[v_1,v_2]],\quad[v_1,[v_2,v_3]],\quad[v_2,[v_2,v_3]],\)\\\\
	\([v_2,[v_1,v_2]],\quad[v_3,[v_2,v_3]],\)
	\item \([v_1,[v_1,[v_1,v_2]]],\quad[v_1,[v_2,[v_1,v_2]]],\quad[v_2,[v_2,[v_1,v_2]]]\),\\\\\([v_2,[v_2,[v_2,v_3]]],\quad[v_2,[v_3,[v_2,v_3]]],\quad[v_3,[v_3,[v_2,v_3]]],\)\\\\\([v_1,[v_1,[v_2,v_3]]],\quad[v_1,[v_2,[v_2,v_3]]],\quad[v_1,[v_3,[v_2,v_3]]],\)\\\\
	\([v_2,[v_1,[v_2,v_3]]].\)
\end{itemize}
We will make use of this basis in the proof of the following proposition.
\begin{proposition}
	Let \(\G=(V,E)\) be a graph with \(\# V>1\) and \(\# E\geq 1.\) Then it follows that if \(\n_{\G,4}\) has a nice basis, \(\G\) cannot have \(\G'\) as a subgraph.
\end{proposition}
\begin{proof}
		Suppose \(\G=(V,E)\) has a subgraph \(\G'=(V',E')\) where \(V'=\{v_1,v_2,v_3\}\) and \(E'=\{\{v_1,v_2\},\{v_2,v_3\}\}\). By \Cref{lem. pi:G->G'} we know that there exists a surjective Lie algebra homomorphism 
	\[\pi:\n_{\G,4}\rightarrow\n_{\G',4}\]
	such that \(\pi(v)=v\) if \(v\in V'\) and \(\pi(v)=0\) if \(v\in V\setminus V'\). 

	Suppose by contradiction that \(\mathcal{B}=\{X_1,\dots,X_n\}\) is a nice basis for \(\n_{\G,4}\). Note that, after possibly reordering, we can suppose that
	\[\{\pi(X_1)+\gamma_2(\n_{\G',4}),\pi(X_2)+\gamma_2(\n_{\G',4}),\pi(X_3)+\gamma_2(\n_{\G',4})\}\]
	forms a basis for \(\n_{\G',4}/\gamma_2(\n_{\G',4}).\)
	Now note that \(\gamma_2(\n_{\G',4})/\gamma_3(\n_{\G',4})\) should be two-dimensional, therefore we can suppose that there are \(\alpha,\beta\in\R\) such that
	\begin{equation}\label{eq. LC-in-ker(pi)}\pi([X_1,X_3]-\alpha[X_1,X_2]-\beta[X_2,X_3])\in\gamma_3(\n_{\G',4}).\end{equation}

	Now define for \(i\in\{1,2,3\}\) the vectors
	\[U_i:=[X_i,[X_1,X_3]],\quad V_i:=[X_i,[X_1,X_2]],\quad W_i:=[X_i,[X_2,X_3]].\]
	It is now clear that the set \[S:=\{\pi(U_i)+\gamma_4(\n_{\G',4}),\pi(V_i)+\gamma_4(\n_{\G',4}),\pi(W_i)+\gamma_4(\n_{\G',4})\mid 1\leq i\leq 3\}\]
	spans the five-dimensional space \(\gamma_3(\n_{\G',4})/\gamma_4(\n_{\G',4})\). Since \(\mathcal{B}\) is assumed to be nice, we know that the dimension of the vector space spanned by \(U_2,V_3\) and \(W_1\) is at most one and therefore the vectors \(\pi(U_2),\pi(V_3)\) and \(\pi(W_1)\) span at most a one-dimensional vector space. 
	
	First assume that we can find \(\gamma,\delta\in\R\) such that  \[\pi(U_2)=\gamma\pi(W_1),\quad \pi(V_3)=\delta\pi(W_1).\] 
	We thus find that \(S\) spans at most a seven-dimensional vector space. Moreover, we know by \eqref{eq. LC-in-ker(pi)} that
	\begin{align*}\pi(U_1)-\alpha\pi(V_1)-\beta\pi(W_1)&\in\gamma_4(\n_{\G',4}),\\
		\gamma\pi(W_1)-\alpha\pi(V_2)-\beta\pi(W_2)&\in\gamma_4(\n_{\G',4}),\\
		\pi(U_3)-\alpha\delta\pi(W_1)-\beta\pi(W_3)&\in\gamma_4(\n_{\G',4}).
	\end{align*} If \((\alpha,\beta)\neq (0,0)\), we can conclude from the above equations that the dimension of the vector space spanned by \(S\) is bounded above by four, which is impossible. 
	
	A similar analysis works if \(\pi(W_1) = 0\), and hence we conclude that \(\alpha=0=\beta\) and thus
	\[[\pi(X_1),\pi(X_3)]\in\gamma_3(\n_{\G',4}).\] By making \(\n_{\G',4}\) Carnot, it then follows that \(\pi(X_1)+\gamma_{2}(\n_{\G',4}),\pi(X_2)+\gamma_{2}(\n_{\G',4})\) and \(\pi(X_3)+\gamma_{2}(\n_{\G',4})\) generate \(\Car(\n_{\G',4})\) and moreover
	\[[\pi(X_1)+\gamma_{2}(\n_{\G',4}),\pi(X_3)+\gamma_{2}(\n_{\G',4})]=0+\gamma_3(\n_{\G',4}).\]
	We can thus find a Lie algebra morphism \(\varphi:\n_{\G',4}\rightarrow \Car(\n_{\G',4})\) such that \(\varphi(v_i)=\pi(X_i)+\gamma_2(\n_{\G',4})\) for \(i\in\{1,2,3\}.\) Since \(\Car(\n_{\G',4})\cong\n_{\G',4}\), their dimensions must be equal. Moreover, as \(\pi(X_1)+\gamma_2(\n_{\G',4}),\pi(X_2)+\gamma_2(\n_{\G',4})\) and \(\pi(X_3)+\gamma_2(\n_{\G',4})\) generate \(\Car(\n_{\G',4})\), it follows that \(\varphi\) is an isomorphism.

	Since \(\mathcal{B}\) is a nice basis we know that
	\[\Span\{\pi([[X_1,X_2],[X_2,X_3]]), \pi([X_2,[X_1,[X_2,X_3]]]),\pi([X_1,[X_2,[X_2,X_3]]])\}\]
	should be at most one-dimensional by the Jacobi identity. Therefore the corresponding vectors in \(\Car(\n_{\G',4})\) can also span at most a one-dimensional vector space. We now obtain a contradiction since we know that \([v_2,[v_1,[v_2,v_3]]]\) and \([v_1,[v_2,[v_2,v_3]]]\) are linearly independent, so they should be mapped by \(\varphi\) to linearly independent vectors. We can thus conclude that \(\n_{\mathcal{G},4}\) does not admit a nice basis.
\end{proof}
Note that a graph \(\mathcal{G}\) which does not contain \(\mathcal{G}'\) as a subgraph must be of the following form
\begin{center}
	
	\begin{tikzpicture}[
		node/.style={circle, fill=black, inner sep=2pt}, 
		label distance=2mm
		]
		\node[node, label=above:{\(v_1\)}] (v1) at (-2,1) {};
		\node[node, label=below:{\(v_2\)}] (v2) at (-2,-1) {};
		\node[node, label=above:{\(v_3\)}] (v3) at (-1,1) {};
		\node[node, label=below:{\(v_4\)}] (v4) at (-1,-1) {};
		\node[node, label=above:{\(v_{n-1}\)}] (vn-1) at (0,1) {};
		\node[node, label=below:{\(v_n\)}] (vn) at (0,-1) {};
		\node[node, label=above:{\(w_1\)}] (w1) at (1,0) {};
		\node[node, label=above:{\(w_2\)}] (w2) at (2,0) {};
		\node[node, label=above:{\(w_m\)}] (wm) at (3,0) {};
		
		\draw (v1) -- (v2);
		\draw (v3) -- (v4);
		\draw (vn-1) -- (vn);
		\path (v3) -- node[auto=false]{\ldots} (vn);
		\path (w2) -- node[auto=false]{\ldots} (wm);
	\end{tikzpicture}
\end{center}
As the 4-step nilpotent Lie algebra associated to this graph is isomorphic to \(\n_{2,4}^k\oplus\R^m\) with \(n=2k\), this Lie algebra clearly admits a nice basis. We thus proved the following result.
\begin{proposition}\label[proposition]{prop. n(G-4) NB}
	A 4-step nilpotent Lie algebra associated to a graph \(\mathcal{G}\) admits a nice basis if and only if it does not contain a subgraph \(\mathcal{G}'=(V',E')\) where \(V'=\{v_1,v_2,v_3\}\) and \(E'=\{\{v_1,v_2\},\{v_2,v_3\}\}\).
\end{proposition}
	\subsection{Higher nilpotency classes}
	For nilpotency class \(5\) we can use a similar argument as in \Cref{prop. 3-cycle}.
	\begin{proposition}
		Let \(\G=(V,E)\) be a graph, then \(\n_{\G,5}\) admits a nice basis if and only if \(E=\emptyset.\)
	\end{proposition}
	\begin{proof}
		The \textit{`if'}-part is trivial since \(\n_{\G,5}\) is abelian if \(E=\emptyset\). 	Conversely, suppose by contradiction that \(\n_{\G,5}\) admits a nice basis \(\{X_1,\dots, X_n\}\) and \(\#E\geq 1\). It then follows that \(\G\) has  a complete subgraph \(\G' \) with two vertices. We find from \Cref{lem. pi:G->G'} that there is a surjective Lie algebra morphism
		\[\pi:\n_{\G,5}\rightarrow\n_{\G', 5}.\]
		Let \(p:\n_{\G',5}\rightarrow\n_{\G',5}/\gamma_2(\n_{\G',5})\) denote the projection morphism. Without loss of generality we can now suppose that \(\{(p\circ\pi)(X_1),(p\circ\pi)(X_2)\}\) forms a basis for \(\n_{\G',5}/\gamma_2(\n_{\G',5})\).
		 
		We now claim that the Lie subalgebra \(\n\) of \(\n_{\G,5}\) generated by \(X_1\) and \(X_2\) is a free \(5\)-step nilpotent Lie algebra. To show this we will use the universal property. Let \(f:\{X_1,X_2\}\rightarrow \mathfrak{h}\) be any map, where \(\mathfrak{h}\) is a \(5\)-step nilpotent Lie algebra. Since \(\pi(X_1)\) and \(\pi(X_2)\) are generators for \(\n_{\G',5}\), it follows that we can find a unique Lie algebra morphism 
		\[\varphi:\n_{\G',5}\rightarrow \mathfrak{h}\] 
		such that for all \(i\in\{1,2\}\) we have that \(\varphi(\pi(X_i))=f(X_i).\) We can now look at the composition \[\varphi\circ \pi:\n\rightarrow \mathfrak{h}.\] Uniqueness follows from the fact that \(\n\) is generated by \(\{X_1,X_2\}.\) 
		
		Since we assumed that \(\{X_1,\dots,X_n\}\) is a nice basis for \(\n_{\G,5}\), we conclude that \(\n\cong\n_{2,5}\) has a nice basis as well which leads to a contradiction.
	\end{proof}
	From the result above we can deduce the following.
	\begin{corollary}
		Let \(\G=(V,E)\) be a graph and \(c\geq 5\), then \(\n_{\G,c}\) admits a nice basis if and only if \(E=\emptyset.\)
	\end{corollary}
	\begin{proof}
	Note that for a fixed graph \(\G\) it holds that
		\[\n_{\G,c}=\frac{\n_{\G,c+1}}{\gamma_{c+1}(\n_{\G,c+1})}.\]
By \Cref{lem. UCS/LCS-adapted} it holds that if a Lie algebra \(\g\) admits a nice basis, then the projection of this basis onto \(\g/\gamma_i(\g)\) gives rise to a nice basis for \(\g/\gamma_i(\g)\). Hence the result follows.
	\end{proof}
	\section*{Acknowledgements}
	We thank \textbf{Jorge Lauret} for his suggestion to the first author to study nice bases on almost abelian Lie algebras.

\end{document}